%% file: main.tex
\documentclass{article}
\usepackage{dirtytalk}
\usepackage[margin=3cm]{geometry}
\usepackage[utf8]{inputenc}
\usepackage[english]{babel}
\usepackage{amsthm}
\usepackage{amsmath}
\usepackage{amssymb}
\usepackage{amsfonts}
\usepackage{xspace}
\usepackage{hyperref}
\usepackage{mathtools}
\usepackage{tikz-cd}
\usepackage{hyperref}
\usepackage[style=alphabetic]{biblatex}
\hypersetup{
    colorlinks=true,
    linkcolor=blue,
    filecolor=magenta,      
    urlcolor=cyan,
}
\usepackage[title]{appendix}
\theoremstyle{definition}
\theoremstyle{plain}
\newtheorem*{theorem*}{Theorem}
\newtheorem{theorem}{Theorem}[section]
\newtheorem{definition}[theorem]{Definition}
\newtheorem{lemma}[theorem]{Lemma}
\newtheorem{corollary}[theorem]{Corollary}
\theoremstyle{remark}
\newtheorem*{remark}{Remark}

\newcommand{\R}{\mathbb{R}}

\newcommand{\Z}{\mathbb{Z}}
\newcommand{\Q}{\mathbb{Q}}

\newcounter{Chapcounter}

\newcommand{\chapter}[1] 
{ {\centering          
  \addtocounter{Chapcounter}{1} \Large \underline{\textbf{ \color{blue} Chapter \theChapcounter: ~#1}} }   
  \addcontentsline{toc}{section}{ \color{blue} Chapter:~\theChapcounter~~ #1}    
}

\title{Differentiability of Volumes and Equidistribution on Quasi-Projective Varieties}
\author{Debam Biswas}
\date{}

\begin{document}

\maketitle
\begin{abstract}
Differentiability of geometric and arithmetic volumes of Hermitian line-bundles leads to the proof of equidistribution results on projective varieties using the variational principle. In this article, we work in the setting of adelic divisors on quasi-projective varieties recently introduced by Xinyi Yuan and Shou-Wu Zhang to show that their geometric and arithmetic adelic volume functions are differentiable on the big cone. We introduce the notions of positive intersection products and show that the derivatives are realised as positive intersection products against integrable divisors at big points. We show an analogue of the Fujita Approximation for restricted volumes of adelic divisors in the geometric setting using a construction similar to those of positive intersections and as an application of our differentiability result we derive a quasi-projective analogue of the equidistribution theorem of Berman and Boucksom which generalises the equidistribution obtained by Yuan and Zhang for arithmetically nef adelic divisors on quasi-projective varieties.
\end{abstract}
\input{classific.tex}
\section*{Notations}
\begin{itemize}
    \item By a \emph{variety} over a field, we always mean an integral separated finite-type scheme over the field.
    \item By an \emph{arithmetic variety} over a Dedekind domain, we mean a projective integral flat scheme over the spectrum over the Dedekind domain.
    \item We generally denote the \emph{arithmetic adelic divisors} over an essentially quasi-projective variety by $\overline{D}, \overline{E}$ etc and \emph{geometric adelic divisors} by $D,E$ etc.
    \item We denote \emph{arithmetic adelic volumes} by $\widehat{\text{vol}}(\cdot)$ and \emph{geometric adelic volumes} by $\text{vol}(\cdot)$.
    \item Given a function $f\colon \Q\rightarrow \R$, we say that $f$ is \emph{differentiable} at $t=0$ if 
    $f'(t):=\lim_{t\to 0}\frac{f(t)-f(0)}{t}$ exists and we say that the \emph{derivative} of $f$ at 0 is $f'(0)$.
    \item By a \emph{nef adelic divisor}, we always denote an \textbf{arithmetically} nef adelic divisor in the sense of \cite[Section 2.4]{yuan2021adelic}.
\end{itemize}
\section{Introduction}
Let $k$ be a field and $X$ a projective variety over $k$. If $L$ is a line bundle on $X$ then we define
\[\text{vol}(L):=\limsup_{m\to \infty}\frac{d!\cdot\text{dim}_k(H^0(X,L^{\otimes m}))}{m^d}\]
where $d=\text{dim}(X)$. We say $L$ is big if $\text{vol}(L)>0$. Boucksom, Favre and Jonsson showed that the function $\text{vol}(L)$ is differentiable on the big cone in \cite{Bouckdiv}. Lazarsfeld and Mustață independently showed the differentiability result by using Okounkov bodies in \cite{lazmus}. The differential of the volume function in a suitable direction involves the definition of \emph{positive intersection products} which were first introduced in \cite{positivoriginal}.\\
In Arakelov geometry, there is a similar notion of arithmetic volumes for Hermitian line bundles which measures the asymptotic growth of \emph{small global sections} of a given Hermitian line bundle. More precisely if $\mathcal{X}\rightarrow \text{Spec}(\mathbb{Z})$ is a projective arithmetic variety of dimension $d$ with generic fiber $X$ and let $\overline{\mathcal{L}}$ be a Hermitian line bundle with generic fiber $L$, then there is an induced Hermitian metric $|\cdot|$ on the complexified line bundle $L_{\mathbb{C}}:=L\otimes_{\mathbb{Q}}\mathbb{C}$. This metric induces a sup-norm $||\cdot||$ corresponding to every global section $H^0(X,L)$ and we define the space of \emph{small sections} as
\[\widehat{H}^0(\mathcal{X},\overline{\mathcal{L}}):=\{s\in H^0(X,L)\mid ||s||\le 1\}\ \text{and}\ \widehat{h}^0(\mathcal{X},\overline{\mathcal{L}}):=\text{log}\#\hat{H^0}(\mathcal{X},\overline{\mathcal{L}})\]
Then we can finally define the \emph{arithmetic volume} of $\overline{\mathcal{L}}$ as
\[\widehat{\text{vol}}(\overline{\mathcal{L}}):=\limsup_{m\to \infty}\frac{d!\cdot\widehat{h}^0(H^0(\mathcal{X},\overline{\mathcal{L}}^{\otimes m}))}{m^d}\]
We call $\overline{\mathcal{L}}$ to be \emph{big} if $\widehat{\text{vol}}(\overline{\mathcal{L}})>0$.
Analogously to the classical volume function in algebraic geometry, Chen showed that the arithmetic volume function is differentiable in the big cone. In \cite{chendiff}, he introduces an arithmetic analogue of the positive intersections in \cite{positivoriginal} and shows that the differentiability at the big point is given by the arithmetic positive intersection with a suitable direction analogously.\\ 
In their recent work \cite{yuan2021adelic}, Yuan and Zhang introduce the notion of \emph{adelic divisors} on a quasi-projective (arithmetic) variety $U$. Their idea is to consider all arithmetic divisors coming from some projective model of the fixed quasi-projective variety $U$ and put a \emph{boundary topology} on them which measures their growth along the boundary outside $U$. Then they define the \emph{adelic divisors} as those divisors which are \say{compactified} with respect to the boundary topology. More precisely if we fix a boundary divisor $\overline{D_0}:=(D_0,g_0)$ \emph{i.e} an arithmetic divisor $\overline{D_0}$ on some projective compactification $X_0$ of $U$ such that $\text{Supp}(D_0)=X_0\backslash U$, then an \emph{adelic divisor} is given by sequence $\{X_i,\overline{D_i}\}$ where $X_i$ are projective compactifications of $U$ and $\overline{D}_i$ are arithmetic divisors on $X_i$ satisfying a Cauchy condition \emph{i.e} there is a sequence of positive rationals $\{q_i\}$ converging to 0 such that 
\[-q_i\overline{D_0}\le \overline{D_j}-\overline{D_i}\le q_i\overline{D_0}\ \text{for all}\ j\ge i\]
where the effectivity relations are understood to hold after passing to a common projective model via birational pull-backs. The notion of adelic divisors are closely related to the notion of \emph{b-divisors} introduced in \cite{bdivprim} (see \cite{anabdivisor} for a nice review on b-divisors).\\ 
Yuan and Zhang also introduce the notions of volumes and arithmetic volumes $\widehat{\text{vol}}(\overline{D})$ for an adelic divisor $\overline{D}$ on a quasi-projective variety and show that they satisfy many properties of the classical volumes on projective varieties like continuity, log-concavity etc (see \cite[Chapter 5]{yuan2021adelic}. Following \cite{Bouckdiv} in the geometric case and \cite{chendiff} in the arithmetic case, we define \emph{positive intersection products} of adelic divisors $\overline{D}$ against integrable directions $\overline{E}$ denoted by $\langle\overline{D}^{d-1}\rangle\cdot \overline{E}$ (see Definition \ref{def:positnt} in section 3). Our first main result in this article is to show that these volume functions are differentiable. Suppose $k=(B,\Sigma)$ is a tuple where $B$ is a number ring or a smooth projective curve with fraction field $K$ and $\Sigma\subseteq \text{Hom}(K,\mathbb{C})$ and suppose $U$ is an essentially quasi-projective variety over $k$. Then in \ref{theorem:finaldiff} we show the following differentiability of adelic volumes
\begin{theorem*}[A]
        Suppose $\overline{D}$ is a big adelic divisor and $\overline{E}$ is an integrable adelic divisor on a normal essentially quasi-projective variety $U$ over $k$ of dimension $d$. Then the function $t\mapsto \widehat{\emph{vol}}(\overline{D}+t\overline{E})$ is differentiable at $t=0$ with derivative given by 
        \[\frac{d}{dt}\widehat{\emph{vol}}(\overline{D}+t\overline{E})\mid_{t=0}=d\cdot\langle\overline{D}^d\rangle\cdot\overline{E}\]
\end{theorem*}
The arguments to deduce the differentiability of both geometric and arithmetic adelic volumes are very similar to the ones used by Boucksom, Favre and Jonsson in \cite{Bouckdiv} to deduce differentiability of geometric volumes (see the proof of \ref{theorem:finaldiff}). The crucial point is using the continuity of positive intersection products and adelic volumes to obtain a limit version of the bounds in \cite{Bouckdiv} while perturbing in small directions of the boundary divisor. We also use Fujita approximation \cite[Theorem 5.2.8]{yuan2021adelic} and Siu's inequality \cite[Theorem 5.2.2(2)]{yuan2021adelic} for adelic volumes.\\
Given a closed sub-variety of a projective variety, there is an invariant called \emph{asymptotic intersection number} defined along it for a line bundle in \cite{asympint}. The definition of this intersection number is closely related to the definition of positive intersection products. In \cite{asympint} it is showed that these intersection numbers are related to the \emph{restricted volumes} of the line bundle along the closed sub-variety provided that the closed sub-variety is \say{general enough} with respect to the bundle. Given a closed subvariety $E$ of a quasi-projective variety and a (geometric) adelic divisor $\overline{D}$ on $U$, we have introduced the notions of the restricted volumes denoted by $\widehat{\text{vol}}_{U|E}(\overline{D})$  and \emph{augmented base locus} of $\overline{D}$ denoted by $B_+(\overline{D})$ in \cite{biswas2023convex}. In this article we introduce an invariant similar to the asymptotic intersection numbers for adelic divisors which we denote by $\langle\overline{D}^d\rangle\cdot E$ and show that they are equal to the restricted volume when the closed sub-variety is not contained in the augmented base locus. This allows us to obtain a version of Fujita approximation for restricted volumes. Now suppose $k$ is a field and $U$ is a quasi-projective variety over $k$. In Theorem \ref{theorem:useless}, we show the following
\begin{theorem*}[B]
    Suppose $D$ is a big adelic divisor and suppose $E$ is a closed sub-variety of $U$ such that $E\nsubseteq B_+(D)$. Then we have 
    \[\emph{vol}_{U|E}(D)=\langle D^{d-1}\rangle\cdot E\]
    where $d=\emph{dim}(U)$
\end{theorem*}

Next as an application of our differentiability results we consider the variational principle useful in proving equidistribution conjectures. We work uniformly over the base $K$ which can either be a number field or a function field of one variable. The differentiability of both the geometric and arithmetic adelic volumes allows us to mimick the arguments in \cite[Section 5.2]{chendiff} to obtain an equdistribution for big arithmetic adelic divisors, Suppose $U$ is now a quasi-projective variety over $K=\text{Frac}(B)$ where $B$ is either the ring of integers of a number field or a smooth projective curve. We fix a place $v$ on $K$ and we denote by $U_v$ the $K_v$-analytic space associated to $U$ where $K_v$ is the completion of $K$ at $v$.  In Theorem \ref{theorem:equidistribution}, we show the following
\begin{theorem*}[C]
    Suppose $U$ is a quasi-projective variety over $K$ of dimension $d§$ and suppose $\overline{D}$ is a big arithmetic adelic divisor in $\widehat{\emph{Div}}(U,k)$ with generic fiber $D$. Furthermore suppose $\{x_m\}$ is a generic sequence of geometric points in $U(\overline{K})$ which is small with respect to $\overline{D}$. Then for any place $v$ on $K$ and for any $g\in C_c^0(U_v)$, we have
     \[\lim_{m\to\infty}\int_{U_v}g\ d\eta_{x_m}=\frac{\langle\overline{D}^d\rangle\cdot\mathcal{O}(g)}{\emph{vol}(D)}\]
     In particular, the sequence of Radon measures $\{\eta_{x_m}\}$ converge weakly to the Radon measure given by 
     \[g\in C_c^0(U_v)\mapsto \frac{\langle\overline{D}^d\rangle\cdot\mathcal{O}(g)}{\emph{vol}(D)}\]
\end{theorem*}
We go on to show that this strengthens the equidistribution theorem for quasi-projective varieties proved by Yuan and Zhang in \cite[Theorem 5.4.3]{yuan2021adelic}. The above theorem can be thought of as an analogue of the equidistribution theorem obtained by Berman and Boucksom in \cite{BermanBoucksom} in the quasi-projective case where we require an extra positivity assumption for the arithmetic adelic divisor to be \textbf{arithmetically} big.\\
The article is organised as follows. In section 2, we start with a brief review of the adelic divisors of Yuan and Zhang following \cite{yuan2021adelic} which will be necessary for our article. In section 3, we start by defining positive intersection products for adelic divisors. We go on to obtain properties of this product and in the end obtain the differentiability result. In section 4 we define \emph{positive intersection products} along a closed sub-variety and relate them with restricted volumes under suitable hypothesis on the closed sub-variety. In section 5, we obtain a fundamental inequality over function fields using the theroy of adelic curves. In section 6, we deduce the equidistribution result as an application of differentiability of geometric and arithmetic adelic volumes and finally in section 7 we relate our equidistribution with two other equidistribution results in the literature.
\section*{Acknowledgements}
The author would like to thank Sébastien Boucksom, Huyai Chen, Walter Gubler, Roberto Gualdi and Antoine Sedillot for  fruitful discussions and comments during the preparation of the article. The author would like to thank François Ballaÿ for pointing out a relation to the equidistribution result of Berman and Boucksom and Xinyi Yuan for pointing towards further generalisations. The author was supported by SFB Higher Invariants during the preparation of this work.
\section{Review of Adelic Divisors}
In this section, we recall the definitions of adelic divisors on quasi-projective varieties over Dedekind bases. We follow the treatment along \cite[Section 2.7]{yuan2021adelic}. We begin by considering a tuple $k:=(B,\Sigma)$ where $B$ is a Dedekind scheme \emph{i.e} a Noetherian integral scheme of dimension 1 with function field $K$ and $\Sigma\subseteq\text{Hom}(\Sigma,\mathbb{C})$. Note that here we allow $\Sigma$ to be the empty set which is typically the case when $B$ is a scheme over a field of positive characteristic. Note that every $\sigma\in\Sigma$ induces an Archimedean valuation $|\cdot|_{\sigma}$ of $\mathbb{C}$ and we have $|\cdot|_{\sigma}=|\cdot|_{\sigma'}$ if and only if $\sigma'=c\circ\sigma$ where $c$ is the complex conugation on $\mathbb{C}$. We further impose that when our scheme $B=\text{Spec}(F)$ for some field $F$ or $B$ is a smooth projective curve, we set that $\Sigma$ is empty, even if $F$ has characteristic 0. We do this to include the geometric case in our uniform terminology. By a variety over $k$ we mean a flat, integral, separated, finite-type scheme $\mathcal{X}$ over $B$. Furthermore we say it is quasi-projective (projective) if $\mathcal{X}$ is quasi-projective (projective) over $B$. For a quasi-projective variety $\mathcal{U}$ over $k$, a \emph{projective model} of $\mathcal{U}$ means a projective variety $\mathcal{X}$ over $k$ together with an open immersion $\mathcal{U}\xhookrightarrow{}\mathcal{X}$.\\
We want to include a slightly more general class of varieties that we want to work with. In order to do so, we recall the notion of \emph{pro-open immersions}. A morphism between integral schemes is said to be \emph{pro-open} if the underlying map of topological spaces is injective and it induces and isomorphism of local rings at each point (see \cite[Section 2.3.1]{yuan2021adelic}). By an \emph{essentially quasi-projective variety} over $k$, we mean an integral, flat, finite scheme $X$ over $B$ together with a projective variety $\mathcal{X}$ over $k$ such that there is a pro-open immersion $X\xhookrightarrow{}\mathcal{X}$. By a \emph{(quasi)projective} model of $X$ over $k$, we mean a (quasi)-projective variety $\mathcal{X}$ over $k$ together with a pro-open immersion $X\xhookrightarrow{} \mathcal{X}$.
Suppose $\mathcal{X}$ is a projective arithmetic variety over $k$. Then we set $\mathcal{X}_{\Sigma}:=\coprod_{\sigma\in\Sigma} \mathcal{X}_{\sigma}$ where $\mathcal{X}_{\sigma}$ denotes the base change of $\mathcal{X}$ to $\mathbb{C}$ via $\sigma$. By an \emph{arithmetic divisor} over $\mathcal{X}$, we mean a tuple $(\mathcal{D},g_{\mathcal{D}})$ where $\mathcal{D}$ is a divisor and $g_{\mathcal{D}}$ is a Green function $\mathcal{X}_{\Sigma}(\mathbb{C})\backslash |\mathcal{D}(\mathbb{C})|\rightarrow \mathbb{R}$ which is invariant under the action $F_{\infty}$ of complex conjugation. This means that if $\overline{\sigma}=c\circ \sigma$, then we require that $g_{\mathcal{D},\overline{\sigma}}=g_{\overline{\mathcal{D}},\sigma}$ where $g_{\overline{\mathcal{D}},\sigma}$ denotes the induced Green's function on the factor $\mathcal{X}_{\sigma}$. We denote the group of arithmetic divisors on $\mathcal{X}$ by $\widehat{\text{Div}}(\mathcal{X})$. For further notions we refer to \cite[Section 2.7]{yuan2021adelic}.\\
Given an open subscheme of $\mathcal{X}$, we denote the \emph{objects of mixed coefficients} as \[\widehat{\text{Div}}(\mathcal{X},\mathcal{U})=\widehat{\text{Div}}(\mathcal{X})_{\mathbb{Q}}\oplus_{\text{Div}(\mathcal{U})_{\mathbb{Q}}}\text{Div}(\mathcal{U})\]
Given a fixed quasi-projective arithmetic variety $\mathcal{U}$ on $k$, we can define the group of \emph{model adelic divisors} as 
\[\widehat{\text{Div}}(\mathcal{U})_{\text{mod}}:=\lim_{\mathcal{X}}\widehat{\text{Div}}(\mathcal{X},\mathcal{U})\]
where we take the filtered colimit by varying $\mathcal{X}$ across all projective models of $\mathcal{U}$ and viewing the objects of mixed coefficients as filtered system via bi-rational pull-backs. We can extend the notions of effectivity to the group of adelic divisors by passing to filtered colimits. By a \emph{boundary divisor} on $\mathcal{U}$, we mean a model divisor $\overline{\mathcal{D}}_0\in\widehat{\text{Div}}(\mathcal{X},\mathcal{U})$ such that $\text{Supp}(\mathcal{D}_0)=\mathcal{X}\backslash \mathcal{U}$. Given the choice of such a boundary divisor we can endow the group of model divisors by a \emph{boundary norm} as 
\[||\cdot||_{\overline{\mathcal{D}}_0}:\widehat{\text{Div}}(\mathcal{U})_{\text{mod}}\rightarrow \mathbb{R}\cup\{\infty\}\]
\[||\mathcal{D}||_{\overline{\mathcal{D}}_0}:=\inf\{q\in \mathbb{Q}\mid -q\overline{\mathcal{D}}_0\le \overline{\mathcal{D}}\le q\overline{\mathcal{D}}_0\}\]
We can define the group \emph{adelic divisors} on $\mathcal{U}$, denoted by $\widehat{\text{Div}}(\mathcal{U},k)$, as the Cauchy completion of $\widehat{\text{Div}}(\mathcal{U})_{\text{mod}}$ with respect to the topology induced by $||\cdot||_{\overline{\mathcal{D}}_0}$ for some boundary divisor $\overline{\mathcal{D}}_0$. Finally given an essentially quasi-projective variety $U$ over $k$, we define the group of adelic divisors on $U$, denoted by $\widehat{\text{Div}}(U,k)$ as
\[\widehat{\text{Div}}(U,k):=\lim_{\mathcal{U}}\widehat{\text{Div}}(\mathcal{U},k)\]
where we take the filtered colimit by varying $\mathcal{U}$ across all quasi-projective models of $U$ over $k$ and using birational pull-backs as transition maps. We further have the notions of \emph{strongly nef}, \emph{nef} and \emph{integrable} adelic divisors denoted by $\widehat{\text{Div}}(U,k)_{\text{snef}}$, $\widehat{\text{Div}}(U,k)_{\text{nef}}$ and $\widehat{\text{Div}}(U,k)_{\text{int}}$ respectively and we refer to \cite[Section 2.4]{yuan2021adelic} for further definitions and details.\\
Note that the notions of effectivity of divisors on $\widehat{\text{Div}}(\mathcal{U})_{\text{mod}}$ induce notions of effectivity on $\widehat{\text{Div}}(\mathcal{U},k)$ and $\widehat{\text{Div}}(U,k)$ by passing to completions and filtered colimits. Given $\overline{D}\in\widehat{\text{Div}}(U,k)$, Yuan and Zhang introduce the notion of \emph{small sections} as
\[H^0(U,\overline{D}):=\{f\in k(U)^{\times}\mid \widehat{\text{div}}(f)+\overline{D}\ge 0\}\]
where $k(U)$ denotes the function field of $U$. They are further able to show that these spaces are finite dimensional in the geometric case and are of finite cardinality in the arithmetic case (see \cite[Lemma 5.1.6]{yuan2021adelic})). This allows them to define the notions of geometric and arithmetic volumes $\widehat{\text{vol}}(\overline{D})$ which will be our main objects of investigation in this article. We say that $\overline{D}\in\widehat{\text{Div}}(U,k)$ is \emph{big} if $\widehat{\text{vol}}(\overline{D})>0$. We refer to \cite[Sections 5.1-5.2]{yuan2021adelic} for details and various properties of adelic volumes that are obtained. 
\section{Positive Intersection products}
 In this section, we define the positive intersection product of (arithmetic)adelic divisors with an integrable adelic divisor inspired by positive intersection products defined in \cite{Bouckdiv} in the geometric case and \cite{chendiff} in the arithmetic case. We go on to show that at big adelic divisors, this product behaves continuously. Then using arguments like in \cite{Bouckdiv} we show that the function $t\mapsto\widehat{\text{vol}}(\overline{D}+t\overline{E})$ is differentiable at $t=0$ and the derivative is given by the defined positive intersection product when $\overline{E}$ is an integrable adelic divisor.\\
 We are going to treat the arithmetic and geometric cases parallely as in \cite{yuan2021adelic}. We will denote by $k:=(B,\Sigma)$ where $B$ is either the ring of integers of a number field $K$, a smooth projective curve with function field $K$ or any field. Furthermore if $B$ is either a smooth projective curve or a field then we force $\Sigma$ to be empty and when $B$ is the ring of integers then we force $\Sigma=\text{Hom}(K,\mathbb{C})$ to be the full set of embeddings. Note that then in any of the above cases, $k$ will be a valued Dedekind scheme as defined in the previous section and hence the formalism of the previous section goes through. We restrict to these cases since absolute intersection numbers are only defined in these global cases in \cite{yuan2021adelic}. We will denote by $U$ an essentially quasi-projective variety over $k$  and by $\mathcal{U}$ any quasi-projective model of $U$ over $k$. Given two such models of $U$, we say one dominates the other if there is a morphism over $U$ between them in the appropriate direction. Note that given two models $\mathcal{U}_1, \mathcal{U}_2$ of $U$ such that $\mathcal{U}_1$ dominates $\mathcal{U}_2$, $\mathcal{U}_1$ is necessarily a birational modification of $\mathcal{U}_2$ which is an isomorphism over $U$.\\
 We have defined adelic divisors on $U$ in the previous section and we give a further description here on a quasi-projective model which will be useful for our computations later. An adelic divisor $\overline{D}$ on a quasi-projective variety $\mathcal{U}$ over $k$ is given by the data $\{\mathcal{X}_i, \overline{D_i}, q_i\}$ where $\mathcal{X}_i$ are projective varieties over $k$ containing $\mathcal{U}$ as a dense open subset, $\overline{D_i}\in\widehat{\text{Div}}(\mathcal{X}_i)_{\mathbb{Q}}\oplus_{\text{Div}(\mathcal{U})_{\mathbb{Q}}}\text{Div}(\mathcal{U})=\text{Div}(\mathcal{X}_i,\mathcal{U})$ are the objects of mixed coefficients as defined in \cite[Section 2.4]{yuan2021adelic} such that $\overline{D_i}$ and $\overline{D_j}$ have the same image in $\text{Div}(\mathcal{U})_{\mathbb{Q}}$ and $\{q_i\}$ is a zero sequence of positive rationals such that they satisfy the \say{Cauchy condition} 
 \[-q_j\overline{D_0}\le \overline{D_i}-\overline{D_j}\le q_j\overline{D_0}\ \text{in}\ \text{Div}(\mathcal{U})_{\text{mod}}\ \text{for all}\ i\ge j\]
 for a boundary divisor $\overline{D_0}$. An adelic divisor $\overline{D}$ is called \emph{(strongly) nef} if $\mathcal{O}(\overline{D})$ is a (strongly) nef adelic line bundle as defined in \cite[Definition 2.5.2]{yuan2021adelic}. Similarly we say an adelic divisor is \emph{integrable} if $\mathcal{O}(\overline{E})$ is an integrable adelic line bundle as defined in \cite[Definition 2.5.2]{yuan2021adelic}. We denote the group of adelic divisors, nef adelic divisors and integrable adelic divisors by $\widehat{\text{Div}}(\mathcal{U},k)$, $\widehat{\text{Div}}(\mathcal{U},k)_{\text{nef}}$ and $\widehat{\text{Div}}(\mathcal{U},k)_{\text{int}}$ respectively.\\
 Finally we can define the group of adelic divisors, nef adelic divisors and integrable adelic divisors on an essentially quasi-projective variety $U$ over $k$ to be
 \[\widehat{\text{Div}}(U,k):=\varinjlim_{\mathcal{U}}\widehat{\text{Div}}(\mathcal{U},k)\ \ \widehat{\text{Div}}(U,k)_{\text{nef}}:=\varinjlim_{\mathcal{U}}\widehat{\text{Div}}(\mathcal{U},k)_{\text{nef}}\ \ \widehat{\text{Div}}(U,k)_{\text{int}}:=\varinjlim_{\mathcal{U}}\widehat{\text{Div}}(\mathcal{U},k)_{\text{int}}\]
 where we vary $\mathcal{U}$ over all the quasi-projective models of $U$ and we take the filtered colimit by viewing the collections of $\widehat{\text{Div}}(\mathcal{U},k)$ as a filtered system via birational pull-backs.\\
 Next we give the main definition of \emph{positive intersection product} of an adelic divisor $\overline{D}$ with a nef adelic divisor $\overline{E}$. Note that Yuan and Zhang has defined top absolute intersection numbers for integrable adelic divisors in the case when $k$ is a field or $\Z$ (see \cite[Prop 4.1.1]{yuan2021adelic}. Our strategy will be to define the positive intersection first for quasi-projective models and then pass to the essentially quasi-projective case. Hence we first give the definition for quasi-projective varieties over $k$. We first start with a definition.
 \begin{definition}
     \label{deinition:freemodel}
     Suppose $\mathcal{X}$ is a projective variety over $k$ and $\overline{A}$ an arithmetic divisor. We say that $\overline{A}$ is free if 
     \begin{itemize}
         \item the underlying line bundle $O(A)$ is semi-ample i.e some positive tensor power of it is generated by global sections when $\Sigma$ is empty.
         \item the associated curvature current of $\overline{A}$ on $\mathcal{X}_{\Sigma}$ is semi-positive and some positive tensor power of the underlying line bundle $O(A)$ is generated by small sections when $\Sigma$ is non-empty.
     \end{itemize}
 \end{definition}
 \begin{remark}
     First we remark that it is easy to check that the definition above can be easily extended to the case of $\Q$-divisors. Furthermore for any open subset $\mathcal{U}$ of $\mathcal{X}$, $\overline{A}$ is model nef on $\mathcal{U}$. This follows from \cite[Prop 2.3]{chendiff} when $\Sigma$ is non-empty and follows from intersection theory for the case when $\Sigma$ is empty. Furthermore note that if $\overline{A}$ is arithmetically ample then it is free. Moreover note that since the notion of freenes is invariant under birational pull-backs, we can talk about free model $\Q$-divisors.
 \end{remark}
\begin{definition}
\label{def:positnt}
Suppose $\mathcal{U}$ is a quasi-projective variety of dimension $d$ over $k$ and $\overline{D},\overline{E}$ adelic divisors on $\mathcal{U}$ such that $\overline{E}$ is nef. Then we define the positive intersection product of $\overline{D}$ with $\overline{E}$, denoted $\langle \overline{D}^{d-1}\rangle\cdot \overline{E}$, as
\[\langle \overline{D}^{d-1}\rangle\cdot \overline{E}=\sup_{X',\overline{A}}\overline{A}^{d-1}\cdot (\pi^*\overline{E})\]
where $(X',\overline{A})$ runs over all tuples such that $X'$ is a projective model of a birational modification $\pi\colon \mathcal{U}'\rightarrow \mathcal{U}$ of $\mathcal{U}$ and $\overline{A}$ is free model $\Q$-divisor on $X'$ such that $\pi^*\overline{D}-\overline{A}\ge 0$ as adelic divisors over $\mathcal{U}'$. We denote any such above tuple $(X',\overline{A})$ as an \emph{admissible approximation} of $\overline{D}$ on $\mathcal{U}$.
\end{definition}

\begin{remark}
\label{remark:hmm}
Note that by \cite[Proposition 4.1.1]{yuan2021adelic}, intersection numbers $\overline{A}^{d-1}\cdot\pi^*\overline{E}$ are defined as $\overline{E}$ (and hence $\pi^*\overline{E}$) is nef and in particular integrable and hence the definition makes sense. Also note that as the notions intersection  numbers and effectivity can be extended to $\mathbb{Q}$-adelic divisors, the same can be done for positive intersection products. Note that for an arbitary $\overline{D}$, the set of its admissible approximations might be empty and in that case we set the positive intersection products to be $-\infty$ as a matter of set theoretic convention. However we almost always exclusively work with a big $\overline{D}$ in which case there always exists an admissible approximation of $\overline{D}$. Moreover by choosing model nef divisors $\overline{\omega}\ge \overline{D}$ it is also easy to see that the intersection numbers $\overline{A}^{d-1}\cdot \pi^*\overline{E}$ are bounded from above uniformly as we vary our admissible approximations and hence the positive intersection is a finite number whenever $\overline{D}$ is big.
\end{remark}
We want to extend the above definition to essentially quasi-projective varieties over $k$. The ordinary intersection products are defined by choosing any quasi-projective model and doing the intersection over there and it is well-defined because of product formula (since any two quasi-projective models of $U$ are birational). We will do something similar for positive intersection products which is the content of our next lemma.

\begin{lemma}
    \label{lemma:positextensionessen}
    Suppose $\pi_0\colon\mathcal{U}_1\rightarrow \mathcal{U}_2$ are two quasi-projective models with dimension $d$ of an essentially quasi-projective variety $U$ over $k$ and with $\pi_0$ a birational morphism. Suppose $\overline{D}\in \widehat{\emph{Div}}(\mathcal{U}_2,k)$ and $\overline{E}\in\widehat{\emph{Div}}(\mathcal{U}_2,k)_{\emph{nef}}$. Then we have
    \[\langle\pi_0^*\overline{D}^{d-1}\rangle\cdot \pi_0^*\overline{E}=\langle\overline{D}^{d-1}\rangle\cdot\overline{E}\]
\end{lemma}
\begin{proof}
    We first suppose that $(\mathcal{U}_2',X_2',\overline{A},\pi)$ is an admissible approximation of $\overline{D}$ on $\mathcal{U}_2$. Since $\mathcal{U}_1$ itself is a birational modification of $\mathcal{U}_2$, we get that the fiber product $\mathcal{U}_1':=\mathcal{U}_2'\times_{\mathcal{U}_2}\mathcal{U}_1$ is a birational modification of $\mathcal{U}_1$ and a pull-back square
    \[\begin{tikzcd}
\mathcal{U}_1' \arrow[r, "\pi_0'"] \arrow[d, "\pi'"] & \mathcal{U}_2' \arrow[d, "\pi"] \\
\mathcal{U}_1 \arrow[r, "\pi_0"]                     & \mathcal{U}_2                  
\end{tikzcd}\]
    Denote by $\pi'\colon\mathcal{U}_1'\rightarrow \mathcal{U}_1$ be the canonical birational morphism. Since $\pi^*\overline{D}\ge \overline{A}$ in $\widehat{\text{Div}}(\mathcal{U}_2',k)$, pulling back by $\pi_0'$ we get that $\pi'^*(\pi_0^*\overline{D})=(\pi_0')^*(\pi^*\overline{A})\ge \pi_0'^*\overline{A}$ where $\pi_0'\colon \mathcal{U}_1'\rightarrow \mathcal{U}_2'$ is the canonical morphism and hence $(\mathcal{U}_1',(\pi_0')^*X_2',(\pi_0')^*\overline{A},\pi')$ is an admissible approximation of $\pi_0^*\overline{D}$ on $\mathcal{U}_1$ since $\pi_0'^*\overline{A}$ is free. The projection formula clearly yields $\pi_0'^*\overline{A}\cdot \pi'^*(\pi_0^*\overline{E})=\overline{A}\cdot\pi^*\overline{E}$. We easily deduce from the definition of positive intersection products that 
    \begin{equation}
        \label{equa:ekdombal}
        \langle\pi_0^*\overline{D}^{d-1}\rangle\cdot\pi_0^*\overline{E}\ge \langle\overline{D}^{d-1}\rangle\cdot\overline{E}
    \end{equation}
    On the other hand, suppose $(\mathcal{U}_1',X_1',\overline{A},\pi)$ is an admissible approximation of $\pi_0^*\overline{D}$ on $\mathcal{U}_1$. Then clearly $\mathcal{U}_1'$ is a birational modification of $\mathcal{U}_2$ as well. Moreover since $\pi^*(\pi_0^*\overline{D})\ge \overline{A}$, we clearly get that $(\mathcal{U}_1',X_1',\overline{A},\pi_0\circ\pi)$ is an admissible approximation of $\overline{D}$ on $\mathcal{U}_2$. This clearly yields 
    \begin{equation}
        \label{equa:ekdombaal2}
        \langle\overline{D}^{d-1}\rangle\cdot\overline{E}\ge \langle\pi_0^*\overline{D}^{d-1}\rangle\cdot\pi_0^*\overline{E}
    \end{equation}
    Equations \eqref{equa:ekdombal} and \eqref{equa:ekdombaal2} together clearly finishes the proof of the claim.
\end{proof}
With the above lemma which shows invariance of positive intersection products with respect to birational pull-backs, we can now extend the definition of our positive intersection products to essentially quasi-projective varieties as follows
\begin{definition}
    \label{def:essenpositint}
    Suppose $U$ is an essentially quasi-projective variety over $k$ such that a quasi-projective model of $U$ has dimension $d$ and $\overline{D}\in\widehat{\emph{Div}}(U), \overline{E}\in \widehat{\emph{Div}}(U)_{\emph{nef}}$. Furthermore suppose $\overline{D}$ arises as an arithmetic adelic divisor $\overline{\mathcal{D}}$ on a quasi-projective model $\mathcal{U}$ of $U$. Then we define the positive intersection product of $\overline{D}$ against $\overline{E}$, denoted by $\langle\overline{D}^{d-1}\rangle\cdot \overline{E}$ as
    \[\langle\overline{D}^{d-1}\rangle\cdot \overline{E}:=\langle\overline{\mathcal{D}}^{d-1}\rangle\cdot \overline{E}\]
    where the right hand side is as defined in Definition \ref{def:positnt}. Note that the definition does not depend on the choice of $\mathcal{U}$ due to Lemma \ref{lemma:positextensionessen}.
\end{definition}
We begin by recording some easy positivity properties of nef adelic divisors and their intersection numbers.
\begin{lemma}
\label{lemma:nef1}
Suppose $\overline{E}_1,\ldots \overline{E}_d$ be nef adelic divisors and suppose $\overline{\omega}$ is a model nef adelic divisor on an essentially quasi-projective variety $U$ over $k$ such that $\overline{\omega}\ge\overline{E}_i$ for all $i$ and suppose any quasi-projective model of $U$ has dimension $d$. Then 
\[0\le \prod_{i=1}^d\overline{E}_i\le \overline{\omega}^d\]
\end{lemma}
\begin{proof}
We assume that all the divisors live in a common quasi-projective model $\mathcal{U}$ of $U$ since intersection numbers are birational invariants. Note that it is easy to check that the absolute intersection numbers defined in \cite[Prop 4.1.1]{yuan2021adelic} is continuous in the sense that for all rational $t$ and integrable adelic divisors $\overline{E}_i, \overline{M}$, $\lim_{t\to 0}\prod_i(\overline{E}_i+t\overline{M})=\prod_i\overline{E}_i$ and hence it is enough to check the above inequalities for strongly nef $\overline{E}_i$'s. Note that the left hand side of the inequality is already obtained in Propositon 4.1.1 and hence it is enough to show the right hand side. Suppose $\overline{E}_i$ is given by the models $\{X_j,\overline{E}_{ij}\}$ and rationals $\{q_j\to 0\}$ such that each $\overline{E}_{ij}$ is nef which we can always assume by passing to finer models.  Then note that we have the effectivity relation $\overline{\omega}+q_j\overline{D}_0\ge \overline{E}_{ij}$ by the Cauchy condition $\overline{E}_{i}\ge\overline{E}_{ij}-q_j\overline{D}_0$. Moreover we can choose a model ample divisor $\overline{D}_0'\ge \overline{D}_0$ and shrinking $\mathcal{U}$ if necessary, we can assume that $\overline{D}_0'$ has support $X_0\backslash\mathcal{U}$. Hence for the proof we can assume that the boundary divisor is nef since intersection numbers do not change if we shrink $U$ (and $\mathcal{U}$) to a dense open subset. Then clearly we have $(\overline{\omega}+q_j\overline{D}_0)^d\ge \prod_{i=1}^d\overline{E}_{ij}$ for all $j$ since each $\overline{E}_{ij}$, $\overline{\omega}$ and $\overline{D_0}$ is nef. Now as indicated in the proof of Proposition 4.1.1 the intersection number $\prod_{i=1}^d\overline{E}_i$ is just given by the limit of the products $\prod_{i=1}^d\overline{E}_{ij}$ as $j$ goes to infinity and hence we conclude 
\[\overline{\omega}^d=\lim_{j\to \infty}(\overline{\omega}+q_j\overline{D}_0)^d\ge \lim_{j\to\infty}\prod_{i=1}^d\overline{E}_{ij}=\prod_{i=1}^d\overline{E}_i\]
where the first equality follows since $q_j\to 0$ by continuity of intersection numbers and this finishes the proof of the lemma.
\end{proof}
\begin{remark}
Note that the above lemma can be used easily to show that the positive intersection product is not $+\infty$ by bounding our adelic divisor by a large nef model divisor. In case that the set of admissible approximations is empty, we set the positive intersection product to be $-\infty$ as a matter of set theoretic convention. However we will later see that it is never the case when $\overline{D}$ is big.
\end{remark}
Next we want to show that the positive intersection product defined above is continuous at big divisors.
\begin{lemma}
\label{lemma:positcont}
Suppose $\overline{D}$ is a big adelic divisor, $\overline{E}$ is a nef adelic divisor and $\overline{F}$ is any adelic divisor on an essentially quasi-projective variety $U$ over $k$ such that any quasi-projective model of it has dimension $d$. Then there is a positive integer $m$ depending only on $\overline{D}, \overline{E}$ and $\overline{F}$ such that 
\[(1-mt)^{d-1}\langle\overline{D}^{d-1}\rangle\cdot\overline{E}\le\langle(\overline{D}+t\overline{F})^{d-1}\rangle\cdot\overline{E}\le(1+mt)^{d-1}\langle\overline{D}^{d-1}\rangle\cdot\overline{E}\]
for all rational $\frac{1}{m}\ge t\ge 0$.
In particular for $t$ rational
\[\lim_{t\to 0}\langle(\overline{D}+t\overline{F})^{d-1}\rangle\cdot\overline{E}=\langle\overline{D}^{d-1}\rangle\cdot\overline{E}\]
\end{lemma}
\begin{proof}
As usual we assume that all the divisors live in a common quasi-projective model $\mathcal{U}$ of $U$. We begin by noting that there is a positive integer $m$ such that $m\overline{D}\pm \overline{F}\ge 0$. Indeed if $\overline{D},\overline{F}$ are represented by a sequence $\{D_j,F_j\}$ on some common models $X_j$ and rationals $q_j\to 0$ as usual, then as $\overline{D}$ is big, thanks to \cite[Theorem 5.2.1(2)]{yuan2021adelic} we can find a $j$ such that $D_j-q_jD_0$ is big since $\{D_j-q_jD_0\}$ is also a sequence of models representing $\overline{D}$. Then by Kodaira's lemma (\cite[Prop 2.2.6)]{lazarsfeld2017positivity}) for the geometric case and \cite[Corollary 2.4(3)]{bigyuan} for the arithmetic case, we can find a positive integer $m$ such that $m(D_j-q_jD_0)-(F_j+q_jD_0)\ge 0$ and $m(D_j-q_jD_0)+(F_j-q_jD_0)\ge 0$.  Now the usual Cauchy effectivity relations $\overline{D}\ge D_j-q_jD_0$ and $F_j-q_jD_0\le\overline{F}\le F_j+q_jD_0$ yield $m\overline{D}-\overline{F}\ge 0$ and $m\overline{D}+\overline{F}\ge 0$ respectively.\\
Next note it is easy to check from the definition of positive intersection products and Lemma \ref{lemma:nef1} that if $\overline{D}_1\le \overline{D}_2$ are two adelic divisors and $\overline{E}$ any nef adelic divisor, then \[\langle\overline{D}_1^{d-1}\rangle\cdot \overline{E}\le \langle\overline{D}_2^{d-1}\rangle\cdot\overline{E}\]
since $\overline{E}$ is nef. Clearly by the choice of $m$ we have the effectivity relations
\[(1-mt)\overline{D}\le\overline{D}+t\overline{F}\le(1+mt)\overline{D}\]
when $0<t<\frac{1}{m}$
Then the L.H.S of the inequality claimed in the lemma easily follows for all rational $0<t<\frac{1}{m}$ and the right hand follows for all rational $t$ by noticing that the positive intersection product is homogeneous with respect to positive scaling and using the above effectivity relations. The L.H.S  for $t\ge\frac{1}{m}$ follows trivially by noting that the positive intersection product is positive.

\end{proof}
As a consequence of the continuity of positive intersection products, we derive two corollaries which will be useful for us. The first one gives an alternate description of the positive intersection products for big divisors in terms of nef divisors instead of free ones.
\begin{corollary}
    \label{corol:ekhihai}
    Suppose $\overline{D}$ is a big adelic divisor and $\overline{E}$ is a nef adelic divisor on a quasi-projective variety $\mathcal{U}$ over $k$. Then we have
    \[\langle\overline{D}^{d-1}\rangle\cdot\overline{E}=\sup_{(X',\overline{A})} \overline{A}^{d-1}\cdot \overline{E}\]
    where the supremum is taken as $(X',\overline{A})$ varies over all tuples such that $X'$ is a projective model of a birational modification $\pi\colon\mathcal{U}'\rightarrow \mathcal{U}$ and $\overline{A}$ is a nef $\Q$-divisor on $X'$ such that $\pi^*\overline{D}\ge \overline{A}$ in $\widehat{\emph{Div}}(\mathcal{U},k)$. 
\end{corollary}
\begin{proof}
    Note that since all free $\Q$-divisors are automatically nef by our remark earlier, we immediately get that $\langle\overline{D}^{d-1}\rangle\cdot\overline{E}\le\sup_{(X',\overline{A})} \overline{A}^{d-1}\cdot \overline{E}$. For the converse inequality, choose any arithmetically ample (or ample in the geometric case) divisor $\overline{\omega}$ on $X'$. Then since $\overline{A}$ is nef, we get that $\overline{A}+\epsilon\cdot\overline{\omega}$ is ample and hence free for any positive rational $\epsilon$. Note that then $(X',\overline{A}+\epsilon\cdot\overline{\omega})$ is an admissible approximation of the big adelic divisor $\overline{D}+\epsilon\cdot\overline{\omega}$. Thus we get by the definition of positive intersection products that 
    \[(\overline{A}+\epsilon\cdot\overline{\omega})^{d-1}\cdot \overline{E}\le \langle(\overline{D}+\epsilon\cdot\overline{\omega})^{d-1}\rangle\cdot \overline{E}\]
    Then letting $\epsilon\to 0$ and noting that positive intersection products are continuous at $\overline{D}$ since it is big (Lemma \ref{lemma:positcont}), we deduce 
    \[\overline{A}^{d-1}\cdot\overline{E}\le \langle\overline{D}^{d-1}\rangle\cdot \overline{E}\]
    Since $(X',\overline{A})$ was an arbitary admissible approximation, taking supremum of the left hand side above across all such approximations easily gives the desired converse inequality and finishes the proof.
\end{proof}
\begin{remark}
    Note that the above lemma shows that we can choose our approximations such that $\overline{A}$ is just nef instead of being free. Hence for the rest of the article, we will freely choose approximations to be either nef or free depending on what is suitable for our argument and this in particular shows almost by definition that for model nef divisors $\overline{D}$, the positive intersection products are the same as usual intersection products.
\end{remark}
As a second corollary of the continuity of positive intersections, we deduce that for nef and big arithmetic adelic divisors, the positive intersection products are the same as usual absolute intersection numbers. 
\begin{corollary}
    \label{corol:positiisusual}
    Suppose $\overline{D}$ is a big and nef arithmetic adelic divisor on $U$ and $\overline{E}$ be any nef arithmetic adelic divisor. Then we have
    \[\langle\overline{\mathcal{D}}^d\rangle\cdot\overline{E}=\overline{D}^d\cdot\overline{E}\]
\end{corollary}
\begin{proof}
    Note that after choosing a common quasi-projective model of $U$, we can assume that $\overline{D}$ is given by a Cauchy sequence $\{\overline{D}_i,q_i\}$ as usual where each $\overline{D}_i$ is model nef. Then we have the effectivity relation
    \[\overline{D}-q_i\overline{D}_0\le \overline{D}_i\le \overline{D}+q_i\overline{D}_0\ \text{for all}\ i\]
    for a sequence of positive rational numbers $q_i$ converging to $0$. Then we have the inequality of positive intersection products
    \[\langle(\overline{D}-q_i\overline{D}_0)^d\rangle\cdot \overline{E}\le \langle\overline{D}_i^d\rangle\cdot\overline{E}\le \langle(\overline{D}+q_i\overline{D}_0\rangle\cdot\overline{E}\]
    Hence we can deduce from Lemma \ref{lemma:positcont} that $\lim_{i\to\infty}\langle\overline{D}_i^d\rangle\cdot \overline{E}=\langle\overline{D}^d\rangle\cdot \overline{E}$ since $\overline{D}$ is big. However since $\{\overline{D}_i\}$ is a sequence of model nef divisors converging to $\overline{D}$ and $\overline{E}$ is nef we have that $\lim_{i\to\infty}\overline{D}_i^d\cdot\overline{E}=\overline{D}^d\cdot\overline{E}$ by how absolute intersection numbers for integrable adelic divisors are constructed in \cite[Prop 4.1.1]{yuan2021adelic}. Now since $\overline{D}_i$ are model nef, by our remark before this Lemma. we have that $\langle\overline{D}_i^d\rangle\cdot \overline{E}=\overline{D}_i^d\cdot \overline{E}$. Putting everything together we conclude
    \[\langle\overline{D}^d\rangle\cdot \overline{E}=\lim_{i\to\infty}\langle\overline{D}_i^d\rangle\cdot\overline{E}=\lim_{i\to\infty}\overline{D}_i^d\cdot\overline{E}=\overline{D}^d\cdot \overline{E}\]
    which finishes the proof.
\end{proof}
We want to show that the intersection products is additive in $\overline{E}$ nef so that we extend the product to integrable $\overline{E}$. For that we first need to show that the family of admissible approximations are filtered under the natural relation of dominance. We essentially use an argument of Chen in \cite{chendiff} but since we work more generally over function fields, we restate it for clarity.
\begin{lemma}
    \label{lemma:free}
    Suppose $X$ is a projective variety over $k$. Suppose $\overline{D}$ is an arithmetic divisor on $X$ and $\overline{A}_1, \overline{A}_2$ are two free $\Q$-arithmetic divisors on $\mathcal{X}$ such that $\overline{D}\ge \overline{A}_i$ for $i=1,2$. Then there is a birational modification $\pi\colon X'\rightarrow X $ and a free arithmetic $\Q$-divisor $\overline{A}$ on $X'$ such that $\overline{A}\ge \pi^*\overline{A}_i$ for $i=1,2$ and $\pi^*\overline{D}\ge \overline{A}$. 
\end{lemma}
\begin{proof}
    Note that the arithmetic case when $\Sigma$ is non-empty is exactly the content of \cite[Proposition 3.1]{chendiff} and the geometric case follows by the same construction by ignoring the part about Hermitian metrics.
\end{proof}
\begin{lemma}
\label{lemma:positlinear}
For nef adelic divisors $\overline{E}_i$, $1\le i\le n$ and any big adelic divisor $\overline{D}$ living in a common quasi-projective model of $U$, we can find a sequence of positive rational number $q_m$ converging to $0$ and a sequence $(X_m',\overline{A_m})$ of admissible approximations of $\overline{D}+q_mD_0$ such that the following conditions hold
\begin{enumerate} 
    \item $\lim_{m\to\infty} \overline{A_m}^d=\widehat{\emph{vol}}(\overline{D})$
    \item $\lim_{m\to\infty}\overline{A_m}^{d-1}\cdot\overline{E_i}=\langle\overline{D}^{d-1}\rangle\cdot\overline{E_i}$ for all $1\le i\le n$.
\end{enumerate}
In particular the positive intersection products are linear i.e
\[\langle\overline{D}^{d-1}\rangle\cdot(\sum_{i=1}^n\overline{E}_i)=\sum_{i=1}^n\langle\overline{D}^{d-1}\rangle\cdot\overline{E}_i\]
\end{lemma}
\begin{proof}
We begin by noting that it suffices to show the existence of such a sequence. Indeed consider the $n+1$ nef adelic divisors $\overline{E}_i$ for $1\le i\le n$ and $\sum_{i=1}^n\overline{E}_i$. Choose a sequence of model free divisors such that it satisfies the two conditions for the above $n+1$ nef divisors. Clearly the linearity is true for ordinary intersections products of $\overline{A_m}$ against the $\overline{E}_i$. Then we can conclude by taking the limit as $m\to\infty$ and using the property 2 of the sequence.\\
To prove the existence of such a sequence, first assume $(X_j,D_j,q_j)$ is a Cauchy sequence of model divisors converging to $\overline{D}$ with respect to the boundary divisor of $D_0$. Now suppose $(X_{ij}',\overline{A}_{ij})$ are sequences of admissible approximations of $\overline{D}$ for $1\le i\le n$ such that 
\[\lim_{j\to\infty}\overline{A}_{ij}^{d-1}\cdot\overline{E}_i=\langle\overline{D}^{d-1}\rangle\cdot\overline{E_i}\]
for each $i$ and suppose $(X_{0j}',\overline{A}_{0j})$ is a sequence of admissibile approximation such that
\[\overline{A_{0j}}^d\ge\widehat{\text{vol}}(\overline{D})-\frac{1}{j}\]
Note that we can find such sequences for $1\le i\le n$ merely by the definition of positive intersection products and for $i=0$ due to Fujita approximation since we assumed $\overline{D}$ is big and ample divisors are free (see \cite[Theorem 5.2.8]{yuan2021adelic}).
Furthermore by passing to finer models that $X_{ij}'=X_j'$ for $0\le i\le n$ and that there are birational morphisms $\pi_j\colon X_j'\rightarrow X_j$. Then note that for $0\le i\le n$, $\overline{A}_{ij}\le \pi_j^*(D_j+q_jD_0)$. By applying Lemma \ref{lemma:free} repeatedly we can find a birational $\pi_j'\colon X_j''\rightarrow X_j'$ and a free $\mathbb{Q}$-divisor $\overline{A_j}\ge \pi_j'^*\overline{A}_{ij}$ on $X_j''$ for $0\le i\le n$ such that $(\pi_j\circ\pi_j')^*(D_k+q_kD_0)-\overline{A}\ge 0$. Then the effectivity relations $D_j+q_jD_0\le \overline{D}+2q_jD_0$ imply that $(X_j'',\overline{A_j})$ is an admissible approximation of $\overline{D}+2q_jD_0$ over $U$. Since $\overline{A}, \overline{A}_i, \overline{E}_i$ are all nef, from Lemma \ref{lemma:nef1} we have
\[\overline{A_j}^d\ge\overline{A_{0j}}^d\ge \widehat{\text{vol}}(\overline{D})-\frac{1}{j}\ \text{and}\  \overline{A_j}^{d-1}\cdot \overline{E}_i\ge \overline{A}_{ij}^{d-1}\cdot \overline{E}_i\ \text{for}\ i=1,2\]
Hence we deduce 
\[\widehat{\text{vol}}(\overline{D}+2q_jD_0)\ge \overline{A_j}^d\ge \widehat{\text{vol}}(\overline{D})-\frac{1}{j}\ \text{and}\ \langle(\overline{D}+2q_jD_0)^{d-1}\rangle\cdot\overline{E}_i\ge \overline{A_j}^{d-1}\cdot\overline{E}_i\ge \overline{A}_{ij}^{d-1}\cdot\overline{E}_i\ \text{for}\ 1\le i\le n\]
Now using the continuity of volume functions (see \cite[Theorem 5.2.9]{yuan2021adelic} and positive intersection products (see Lemma \ref{lemma:positcont}), taking limits of the above inequality as $j\to\infty$
\[\widehat{\text{vol}}(\overline{D})=\lim_{j\to\infty}\overline{A_j}^d\ \text{and}\ \langle\overline{D}^{d-1}\rangle\cdot \overline{E}_i\ge\lim_{j\to\infty} \overline{A_j}^{d-1}\cdot \overline{E_i}\ge\lim_{j\to\infty} \overline{A_{ij}}^{d-1}\cdot \overline{E_i}=\langle\overline{D}^{d-1}\rangle\cdot \overline{E_i}\ \text{for}\ 1\le i\le n\]
since $q_j\to 0$ as $\to\infty$.
Taking limits of the two above inequalities as $j\to\infty$, we see that the sequence of model free $\Q$-divisors $\{(X_j'',\overline{A_j})\}$ and positive rationals $\{2q_j\}$ does the job.
\end{proof}

\begin{remark}
Note that the linearity described above allows us to extend the definition of positive intersection product of any big adelic divisor $\overline{D}$ with an integrable adelic divisor $\overline{E}$ which we also denote as $\langle\overline{D}^{d-1}\rangle\cdot \overline{E}$.
\end{remark}
 We want obtain the main inequality required for the differentiability where we estimate terms of order at least 2. However for later purpose we want to keep track of the constants involved in the estimates and for that we need an easy lemma which we record next.
\begin{lemma}
    \label{lemma:diffeff}
    Suppose $\overline{E}$ is integrable. Then we can write $\overline{E}=\overline{E}_1-\overline{E}_2$ where $\overline{E}_i$ is nef and effective for $i=1,2$.
\end{lemma}
\begin{proof}
    As usual we choose a common quasi-projective model $\mathcal{U}$ of $U$ where all the divisors live. By definition of integrable divisors we can choose $\overline{E}_i'$ nef such that $\overline{E}=\overline{E}_1'-\overline{E}_2'$. Choose models $E_1'$ and $E_2'$ on a common projective model $X$ of $U$ such that $\overline{E}_i'\ge E_i'$. Then we can choose an ample divisor $\overline{A}_0$ on $X$ such that $E_i'+\overline{A}_0$ is effective for $i=1,2$ by Serre's finiteness theorem. Then the effectivity relation shows $\overline{E}_i'+\overline{A}_0\ge 0$ and nef and hence we deduce the claim for $\overline{E}_i=\overline{E}_i'+\overline{A}_0$.\\
\end{proof}
Next we state the main inequality required for our differentiability.
\begin{lemma}
\label{lemma:mannewlemm}
 Suppose $\overline{A}$ is a free model $\Q$-divisor on $U$ and suppose $\overline{E}=\overline{E}_1-\overline{E}_2$ is an integrable adelic divisors with $\overline{E}_i$ nef.
 Then we have
\[\widehat{\emph{vol}}(\overline{A}+t\overline{E})\ge \overline{A}^d+(d\overline{A}^{d-1}\overline{E})\cdot t-Ct^2\]
for all rational $0\le t\le 1$ where $C=2^dd^2\overline{\omega}^d$ for any model nef divisor $\overline{\omega}$ such that $\overline{\omega}\ge \overline{A}$ and $\overline{\omega}\ge \overline{E}_i$ for $i=1,2$.
\end{lemma}
\begin{proof}
Choose a model nef divisor $\overline{\omega}$ such that $\overline{\omega}\ge \overline{A}$ and $\overline{\omega}\ge\overline{E}_i$ for $i=1,2$. First note that by Siu's inequality (\cite[Theorem 5.2.2(2)]{yuan2021adelic}) we obtain
\begin{equation}
\label{eqn:newa1}
\widehat{\text{vol}}(\overline{A}+t\overline{E})=\widehat{\text{vol}}(\overline{A}+t\overline{E}_1-t\overline{E}_2)\ge(\overline{A}+t\overline{E}_1)^d-d(\overline{A}+t\overline{E}_1)^{d-1}\cdot(t\overline{E}_2)\ \text{for all positive rational}\ t
\end{equation}
noting that $\overline{A}+t\overline{E}_1$ and $t\overline{E}_2$ are nef when $t>0$. Next we have the binomial expansion of intersection numbers
\begin{equation}
\label{eqn:new3}
(\overline{A}+t\overline{E}_1)^d-d(\overline{A}+t\overline{E}_1)^{d-1}\cdot(t\overline{E}_2)=\overline{A}^d+(d\overline{A}^{d-1}\cdot\overline{E})\cdot t-\sum_{k=1}^{d-1}B_k(d\overline{A}^{d-1-k}\overline{E}_1^k\overline{E}_2)t^{k+1}+\sum_{k\ge 2}C_k(\overline{A}^{d-k}\overline{E}_1^k)t^k
\end{equation}
where $B_k$ and $C_k$ are appropriate binomial coefficients all bounded by $2^d$.
Since by choice we have $\overline{\omega}\ge \overline{A}$ and $\overline{\omega}\ge \overline{E}_i$ for $i=1,2$, by Lemma \ref{lemma:nef1} we can bound each coefficient of $t^k$ in the above sum for $k\ge 2$ by $2^d\overline{\omega}^d$. Since all the terms of the last summand above are positive, ignoring them we have the lower bound 
\begin{equation}
\label{eqn:new4}
(\overline{A}+t\overline{E}_1)^d-d(\overline{A}+t\overline{E}_1)^{d-1}\cdot(t\overline{E}_2)\ge \overline{A}^d+(d\overline{A}^{d-1}\cdot\overline{E})\cdot t-Ct^2\ \text{for all}\ 0\le t \le 1
\end{equation}
by counting the number of such higher order terms in \eqref{eqn:new3} which together with $\eqref{eqn:newa1}$ completes the proof.

\end{proof}
We are ready to state and prove the main theorem of differentiability in this section
\begin{theorem}
\label{theorem:finaldiff}
Suppose $\overline{D}$ is a big adelic divisor and $\overline{E}$ is an integrable adelic divisor on a normal essentially quasi-projective variety $U$ over $k$ of dimension $d-1$. Then the function $t\mapsto \widehat{\emph{vol}}(\overline{D}+t\overline{E})$ is differentiable at $t=0$ with derivative given by 
\[\frac{d}{dt}\widehat{\emph{vol}}(\overline{D}+t\overline{E})\mid_{t=0}=d\cdot\langle\overline{D}^{d-1}\rangle\cdot\overline{E}\]
\end{theorem}
\begin{proof}
Since volumes and intersection numbers are birational invariants, we harmlessly omit the birational pull-backs by abuse of notation. We begin by choosing a sequence of free model $\Q$-divisors $\{\overline{A_m}\}$ and positive rationals $q_m$ converging to $0$ which satisfies the conditions of Lemma \ref{lemma:positlinear} for the big divisor $\overline{D}$ and nef effective divisors $\overline{E_i}$ for $i=1,2$ such that $\overline{E}=\overline{E_1}-\overline{E_2}$  and we assume $q_m<A$ for all $m$ and for some positive rational $A$.  We choose a model nef divisor $\overline{\omega}\ge \overline{D}+AD_0$ and $\overline{\omega}\ge \overline{E_i}$ for $i=1,2$. We claim that the following inequality holds:
\begin{equation}
    \label{eqn:beshibeshi2}
    \widehat{\text{vol}}(\overline{D}+t\overline{E})-\widehat{\text{vol}}(\overline{D})\ge(d\langle\overline{D}^{d-1}\rangle\cdot\overline{E})\cdot t-Ct^2
\end{equation}
where $C=2^dd^2\overline{\omega}^d$. Indeed since $\overline{\omega}\ge \overline{D}+AD_0\ge \overline{D}+q_mD_0\ge \overline{A_m}$ and $\overline{\omega}\ge \overline{E_i}$, thanks to Lemma \ref{lemma:mannewlemm} we have the inequality
\[\widehat{\text{vol}}(\overline{D}+q_mD_0+t\overline{E})\ge\widehat{\text{vol}}(\overline{A_m}+t\overline{E})\ge \overline{A_m}^d+\overline{A_m}^{d-1}\cdot\overline{E_1}-\overline{A_m}^{d-1}\cdot\overline{E_2}-Ct^2\]
Now taking limits of the above inequality as $m\to\infty$, we obtain the inequality \eqref{eqn:beshibeshi2} noting the property of the sequence $\overline{A_m}$ and linearity of positive intersection products (Lemma \ref{lemma:positlinear}). Replacing $\overline{E}$ by $-\overline{E}$, we obtain the inequality
\begin{equation}
    \label{eqn:beshibeshi3}
    \widehat{\text{vol}}(\overline{D})-\widehat{\text{vol}}(\overline{D}-t\overline{E})\le (d\langle\overline{D}^{d-1}\rangle\cdot\overline{E}+Ct^2\
\end{equation}
We want to apply the inequality \eqref{eqn:beshibeshi2}  with the flips $\overline{D}\longleftrightarrow\overline{D}+t\overline{E}$ and $\overline{E}\longleftrightarrow-\overline{E}$  and the inequality \eqref{eqn:beshibeshi3} with the flips $\overline{D}\longleftrightarrow\overline{D}-t\overline{E}$ and $\overline{E}\longleftrightarrow-\overline{E}$ for $0<t\ll 1$ small enough positive rational such that $\widehat{\text{vol}}(\overline{D}\pm t\overline{E})>0$ which can be obtained by continuity of adelic volumes (\cite{yuan2021adelic}, Theorem 5.2.9).  Note that $2\overline{\omega}\ge \overline{D}+t\overline{E}_1\ge \overline{D}+t\overline{E}$ for $0\le t\le 1$ since $\overline{E}_i$ are effective. Replacing $\overline{E}$ by $-\overline{E}$ we also deduce $2\overline{\omega}\ge \overline{D}-t\overline{E}$ for $0\le t\le 1$. Hence we change the constant $C$ to $2^dd^2(2\overline{\omega})^d=2^d\cdot C$. In other words we have
\begin{equation}
\label{eqn:nek4}
\widehat{\text{vol}}(\overline{D})-\widehat{\text{vol}}(\overline{D}-t\overline{E})\ge(d\langle(\overline{D}-t\overline{E})^{d-1}\rangle\cdot \overline{E})\cdot t-2^dCt^2\ \  \ \ 0\le t\ll 1
\end{equation}
\begin{equation}
\label{eqn:nek5}
\widehat{\text{vol}}(\overline{D}+t\overline{E})-\widehat{\text{vol}}(\overline{D})\le(d\langle(\overline{D}+t\overline{E})^{d-1}\rangle\cdot \overline{E})\cdot t+2^dCt^2\ \ \ \   0\le t\ll 1
\end{equation}
for $C=2^{d}d^2\overline{\omega}^d$. Finally the claim follows from the equations \eqref{eqn:beshibeshi2}, \eqref{eqn:beshibeshi3}, \eqref{eqn:nek4} and \eqref{eqn:nek5} together with continuity described in Lemma \ref{lemma:positcont}.
\end{proof}
We end the section by deducing the \emph{isoperimetric inequality} for adelic divsors which is the analogue of  \cite[Proposition 4.5]{chendiff}).
\begin{corollary}
\label{corol:isoperimetric}
Suppose $\overline{D}$ and $\overline{E}$ are adelic divisors such that $\overline{D}$ is big and $\overline{E}$ is integrable and effective. Then we have 
\[\langle\overline{D}^{d-1}\rangle\cdot \overline{E}\ge\widehat{\emph{vol}}(\overline{D})^{\frac{d-1}{d}}\widehat{\emph{vol}}\cdot(\overline{E})^{\frac{1}{d}}\]
where $d$ is the dimension of any quasi-projective model of $U$.
\end{corollary}
\begin{proof}
By Theorem \ref{theorem:finaldiff} we have
\begin{equation}
    \label{eqn:bishaldrunk}
    \lim_{t\to 0}\frac{\widehat{\text{vol}}(\overline{D}+t\overline{E})-\widehat{\text{vol}}(\overline{D})}{t}=d\cdot\langle\overline{D}^{d-1}\rangle\cdot \overline{E}
\end{equation}
Moreover since $\overline{E}$ is effective, by \cite[ Theorem 5.2.6]{yuan2021adelic} we have
\[\widehat{\text{vol}}(\overline{D}+t\overline{E})\ge (\widehat{\text{vol}}(\overline{D})^{\frac{1}{d}}+t\widehat{\text{vol}}(\overline{E})^{\frac{1}{d}})^d\ \text{for all positive rational}\ t\]
Using the binomial expansion in the above inequality we get
\[\widehat{\text{vol}}(\overline{D}+t\overline{E})-\widehat{\text{vol}}(\overline{D})\ge (d\cdot\widehat{\text{vol}}(\overline{D})^{\frac{d-1}{d}}\cdot \widehat{\text{vol}}(\overline{E})^{\frac{1}{d}})\cdot t+O(t^2)\]
which together with the equation \eqref{eqn:bishaldrunk} yields the claim by taking the limit as $t\to 0$.

\end{proof}
\section{Relation with restricted volumes in the geometric case}
In this section we define a slight variant of the positive intersection product that we defined in the last chapter inspired by the \emph{asymptotic intersection numbers} defined in \cite{asympint}. We show using very similar arguments as in the previous chapter that these intersection numbers are continuous as well on big divisors. As a consequence we obtain a a \say{generalized Fujita approximation} like theorem for restricted volumes in the case that the sub-variety is not contained in the augmented base locus as an adelic version of \cite[Theorem 2.13]{asympint}.\\
We start working in the the geometric setting assuming $U$ is now a \textbf{normal} quasi-projective variety over a field $k$. Furthermore suppose $D$ is an adelic divisor and $E$ is a closed sub-variety of $U$ of dimension $d+1$. Furthermore suppose $E\xhookrightarrow{} U$ be a closed sub-variety. Then we consider a slight variant of the admissible approximations that we considered before.
\begin{definition}
    \label{def:Eadmissble}
    Suppose $D$ and $E$ are as described above. Then we define an \emph{$E$-admissible approximation} of $D$ with respect to $U$ by the data $(U',X',A,\pi)$ where $\pi\colon U'\rightarrow U$ is a birational morphism such that it is an isomorphism over the generic point of $E$, $X'$ is a projective model of $U'$ and $A$ is a nef $\mathbb{Q}$-divisor on $X'$ such that $\pi^*D-A\ge 0$. We generally abbreviate the whole data by $(X',A)$.
\end{definition}
We have defined the \emph{augmented base locus} $B_+(D)$ and the \emph{the restricted volume} $\text{vol}_{U|E}(D)$ along $E$ (see \cite[Def 2.4]{biswas2023convex} and \cite[Def 2.6]{biswas2023convex}). Furthermore we have showed that if $E\nsubseteq B_+(D)$ and $D$ is given by models $\{X_i,D_i\}$ as usual, then we have $\text{vol}_{U|E}(D)=\lim_{i\to\infty}\text{vol}_{X_i|E_i}(D_i)$ where $E_i$ is the Zariski closure of $E$ in $X_i$ and the terms in the right hand side of the equality denote the classical projective restricted volumes as defined in \cite[Definition 2.1]{asympint} viewing $D_i$ as a $\mathbb{Q}$-divisor in $X_i$ (see\cite[Corollary 2.17]{biswas2023convex}). It is easy to check using the Zariski density of $U$ in each $X_i$ that $\text{vol}_{U|E}(D_i)=\text{vol}_{X_i|E_i}(D_i)$ where in the left we view $D_i$ as model adelic divisors. We next define a positive intersection product of a divisor $D$ along $E$.
\begin{definition}
\label{def:asymposit}
Suppose $D$ and $E$ are as described above. We define
\[\langle D^d\rangle\cdot E=\sup_{(X',A)}A^d\cdot \tilde{E}\]
where $(\pi\colon X'\rightarrow X, A)$ varies over all $E$-admissible approximations of $D$ on $U$ and $\tilde{E}=\pi^{-1}(E)$.
\end{definition}
\begin{remark}
Note that since we assume $\pi\colon U'\rightarrow U$ to be an isomorphism over the generic point of $E$, the strict transforms $\tilde{E}$ always have the same Krull dimension as $E$ and hence the intersection numbers appearing in the definition actually make sense. Note further that it is easy to check that the above product is homogenous with respect to positive scaling in $D$ just as for positive intersection products defined above. Hence we can extend the notion of the above product also to $\mathbb{Q}$-adelic divisors.
\end{remark}
We begin by showing that this product is continuous in the sense of Lemma \ref{lemma:positcont}. The proof will be very similar and hence we mention only the crucial points.
\begin{lemma}
\label{lemma:asympcont}
Suppose $D$ is a big adelic divisor and $\overline{F}$ is any adelic divisor on $U$. Then there is a positive integer $m$ depending only on $D$ and $ \overline{F}$ such that 
\[(1-mt)^d\langle D^d\rangle\cdot E\le\langle(D+tF)^d\rangle\cdot E\le(1+mt)^d\langle D^d\rangle\cdot\ E\]
for all rational $\frac{1}{m}>t\ge 0$.
In particular for $t$ rational
\[\lim_{t\to 0}\langle(D+tF)^d\rangle\cdot E=\langle D\rangle^d\cdot E\]
\end{lemma}
\begin{proof}
Note first that since $D$ is big, using the same argument as in the beginning of the proof of Lemma \ref{lemma:asympcont} we can find a positive integer $m$ depending on $D$ and $F$ such that $mD\pm F\ge 0$. Next note that it is easy to check from the definition of the positive intersection that if $D_1\le D_2$ are two adelic divisors then $\langle D_1^d\rangle\cdot E\le \langle D_2^d\rangle\cdot E$. Now noting the sequence of effectivity relations 
\[(1-mt)D\le D+tF\le(1+mt)D\ \text{for all positive rational}\ t\]
and noting the homogneity of positive intersections with positive scaling we obtain the claim readily for all rational $0\le t<\frac{1}{m}$.

\end{proof}
Next suppose $D$ is a $\mathbb{Q}$-divisor on a projective model $X$ of $U$ and suppose $\overline{E}$ is the Zariski closure of $E$ in $X$. Then we can consider the \emph{asymptotic intersection number} $||D^d\cdot \overline{E}||$ as defined in \cite[Definition 2.6]{asympint} Furthermore using the alternate description of this product when $\overline{E}\nsubseteq \overline{B_+}(D)$ as described in  \cite[Prop 2.11]{asympint} and our definition in \ref{def:asymposit}  we have $\langle D\rangle^d\cdot E=||D^d\cdot \overline{E}||$ provided $\overline{E}\nsubseteq \overline{B_+}(D)$. Here in the left we consider $D$ as an model adelic $\mathbb{Q}$-divisor and $\overline{B_+}(D)$ denotes the classical augmented base locus as defined in \cite[Section 1]{asympint}.  With this observation and using the continuity we show that positive intersection at a big adelic divisor is given by limit of positive intersections of models after perturbing them a little by the boundary divisor.
\begin{lemma}
\label{lemma:metaconv}
Suppose $D$ is a big adelic divisor given by a sequence $\{X_i,D_i,q_i\}$ as usual and furthermore suppose that $E\nsubseteq B_+(D)$. Then we have 
\[\langle D^d\rangle\cdot E=\lim_{i\to\infty}\langle(D_i+q_iD_0)^d\rangle\cdot E=\lim_{i\to\infty}||(D_i+q_iD_0)^d\cdot \overline{E}_i||\]
where $\overline{E}_i$ is the Zariski closure of $E$ in $X_i$.
\end{lemma}
\begin{proof}
Note first that since $D\le D_i+q_iD_0$ we have that $B_+(D_i+q_iD_0)\subseteq B_+(D)$ which in particular implies that $E\nsubseteq B_+(D_i+q_iD_0)$. Now using Zariski density of $U$ in each $X_i$ we conclude that $\overline{E}_i\nsubseteq \overline{B_+}(D_i+q_iD_0)$ now looking at the models as $\mathbb{Q}$-divisors on projective varieties. Hence we conclude from the discussion before the lemma that $\langle(D_i+q_iD_0)^d\rangle\cdot E=||(D_i+q_iD_0)^d\cdot \overline{E}_i||$ for all $i$. Thus we just need to  prove the first equality in the above claim. Note that the effectivity relations $D\le D_j+q_jD_0\le D+2q_jD_0$
yields the inequalities 
\[\langle D^d\rangle\cdot E\le \langle(D_i+q_iD_0)^d\rangle\cdot E\le \langle(D+2q_jD_0)^d\rangle\cdot E\]
by the second point listed in the proof of the previous lemma. This then evidently yields the claim by the continuity and as $q_j\to 0$ by noting that $D$ is assumed big.
\end{proof}
Finally we can state the main theorem of this section which is a variant of Fujita approximation for restricted volumes.
\begin{theorem}
\label{theorem:useless}
Suppose $D$ is a big adelic divisor and suppose $E$ is a closed sub-variety of $U$ such that $E\nsubseteq B_+(D)$. Then we have 
\[\emph{vol}_{U|E}(D)=\langle D^d\rangle\cdot E\]
\end{theorem}
\begin{proof}
We noted in the proof of the previous lemma that $\overline{E}_i\nsubseteq \overline{B_+}(D_i+q_iD_0)$ for all $i$. Hence by \cite[Theorem 2.13]{asympint} we conclude that $\text{vol}_{X_i|\overline{E}_i}(D_i+q_iD_0)=||(D_i+q_iD_0)^d\cdot \overline{E}_i||$. Furthermore since $E\nsubseteq B_+(D)$ we have shown earlier that $\text{vol}_{U|E}(D)=\lim_{i\to\infty}\text{vol}_{X_i|\overline{E}_i}(D_i+q_iD_0)$ since $\{D_i+q_iD_0\}$ is a Cauchy sequence of models representing $D$. Thus combing the above with Lemma \ref{lemma:metaconv} we deduce the claim of the theorem.
\end{proof}
\begin{remark}
Note that when $D$ is a nef $\mathbb{Q}$-divisor we have $\langle D\rangle^d\cdot E=D^d\cdot E=\text{vol}_{U|E}(D)$ and hence the above theorem states that for general enough sub-varieties $E$ with respect to $D$, the restricted volume is \say{approximated} by restricted volume of nef divisors dominated by it which is analogous to the usual Fujita approximation for volumes. A point to note here is that just like we assume the divisor to be big for usual Fujita approximation \emph{i.e} we require $\text{vol}(D)>0$, we require $E\nsubseteq B_+(D)$ for restricted Fujita approximation to hold above and that readily implies $\text{vol}_{U|E}(D)>0$. A natural question to ask would be whether the converse is true \emph{i.e} whether $\text{vol}_{U|E}(D)>0\Rightarrow E\nsubseteq B_+(D)$. This is always the case in the projective setting which is obtained as the main theorem in \cite[Theorem (C)]{asympint}. Unfortunately the proof heavily uses cohomological methods via separation of jets and it is not clear how to extend these methods in the quasi-projective setting.
\end{remark}
\section{Fundamental inequality in the Function Field case}
In this section, we derive a version of the fundamental inequality for arithmetic adelic divisors analogous to \cite[Lemma 5.3.4]{yuan2021adelic}. The result is only stated in the case $K$ is a number field in \cite{yuan2021adelic} and the proof uses Minkowski's techniques from geometry of lattices. Since over a function field of one variable, there are no Archimedean places these methods can not be applied. However we will use the more general theory of adelic curves as in \cite{chenmoriwaki}. Note that given a function field $K$ of a smooth projective curve $B$ over an arbitary field $k$, we can impose the structure of an adelic curve on it. More precisely, it is done by considering the set of closed points of $B$ with the discrete $\sigma$-algebra as a measure space where each singleton $\{v\}$ for a closed point $v$ is endowed with mass $[k(v):k]$. We furthermore impose that the base field $k$ is countable. This is done to ensure that the set of closed points $v$ is countable and thus the discrete $\sigma$-algebra generated by singletons consists of all subsets of $\Sigma$. Thus we can safely ignore any  questions of measurability of functions defined on $\Sigma$. This structure makes $K$ into a proper adelic curve in the sense of \cite{chenmoriwaki}. Furtermore we denote the set of places over $K$ equipped with the usual normalisation by $\Sigma$.\\
Given arithmetic adelic divisor $\overline{D}\in \widehat{\text{Div}}(U,k)$ with generic fiber $D\in \widehat{\text{Div}}(U,K)$, we can associate a metric $|\cdot|_v$ and correspondingly a sup-norm $||\cdot||_v$ over the Berkovich analytification $U_v$. Note that since $U_v$ is not necessarily compact, these sup.norms might be infinite. However we concentrate only on those sections which take finite sup-norm and we define
\[V_m^{\overline{D}}:=\{s\in H^0(U,mD)\mid ||s||_v<\infty\ \text{for all}\ v\in \Sigma,\ ||s||_v\le 1\ \text{for almost all}\ v\in \Sigma\}\]
When the adelic divisor $\overline{D}$ is clear we denote this graded linear series simply by $\{V_m\}$. Then each $V_m$ comes with a family of norms $\{||\cdot||_v\}_{v\in\Sigma}$ for each $v\in \Sigma$ and we denote this data by $\overline{V_m}$ which forms a \emph{norm family} on $V_m$ in the sense of \cite[Section 4.1.1]{chenmoriwaki}. We begin by showing that each $\overline{V_m}$ becomes an \emph{adelic vector bundle} over $K$ with its structure of an adelic curve in the sense of Chen and Moriwaki (\cite[Def 4.1.28]{chenmoriwaki}).
\begin{lemma}
    \label{lemma:vectorbundleache}
    Suppose $\overline{D}\in\widehat{\emph{Div}}(U,k)$ is an arithmetic adelic divisor and $\overline{V_*}$ be its associated graded linear series along with the norm families $\{||\cdot||_{\overline{D},v,m}\}$ on each graded piece $V_m$. Then each norm family $\overline{V_m}$ is an adelic vector bundle on $K$.
\end{lemma}
\begin{proof}
    We begin by noting that the function $v\mapsto ||s||_{\overline{D},v,m}$ is measurable for all $s\in V_m$ since $k$ and hence $\Sigma$ is countable. Moreover by definition we have $||s||_{\overline{D},v,m}\le 1$ for almost all $v$ which clearly shows that the function $v\mapsto \log ||s||_{\overline{D},v,m}$ is upper dominated.\\
    For the upper dominancy of the dual metric, after choosing a quasi-projective model $\cal{U}$ of $U$ we can find  a model divisor $\overline{E}$ on some projective model of $\cal{U}$ such that $\overline{D}\le \overline{E}$. This induces an inclusion of vector spaces
    \[f\colon V_m=H^0(U,mD)\xhookrightarrow{} H^0(U,mE)=E_m\]
    Furthermore if we equip $V_m$ and $E_m$ with their corresponding sup-norms $||\cdot||_{\overline{D},v,m}$ and $||\cdot||_{\overline{E},v,m}$ at each place $v$, then the effectivity relation $\overline{D}\le \overline{E}$ shows that the operator norm of the induced map $f_v$ has $||f_v||\le 1$. Then there is an induced surjective dual map 
    \[f^{\wedge}\colon E_m^{\wedge}\rightarrow V_m^{\wedge}\]
    and furthermore we have by \cite[Prop 1.1.22]{chenmoriwaki} that $||f_v^{\wedge}||\le ||f_v||\le 1$. Thus for any $s\in V_m^{\wedge}$, we deduce 
    \begin{equation}
        \label{equa:lewra}
        \log||s||_{\overline{D},v,m}^{\wedge}\le \log||s'||_{\overline{D},v,m}^{\wedge}
    \end{equation}
    for any $s'\in E_m^{\wedge}$ with $f^{\wedge}(s')=s$. Since $\overline{E}$ is a model divisor, we know by the projective case in \cite[Theorem 6.1.13]{chenmoriwaki} that the function $v\mapsto \log||s'||_{\overline{E},v,m}$ is upper dominated which together with \eqref{equa:lewra} finishes the proof.
\end{proof}
Next we show that the graded linear series $\{V_m\}$ associated to an arithmetic adelic divisor $\overline{D}$ with a big generic fiber $D$ has the same geometric volume as of the full series of $D$.
Note that given a graded linear series $\{V_m\}$ of a geometric adelic divisor $\widehat{\text{Div}}(U,K)$, we can define its volume as
\[\text{vol}(V_m):=\limsup_{m\to\infty}d!\cdot \frac{\text{dim}_K(V_m)}{m^d}\]
Then we can state our next lemma which shows that those sections which have finite sup-norm and \say{local coherence} actually grow just as fast as the full series.
\begin{lemma}
    \label{lemma:ampleser}
    Suppose $\overline{D}\in \widehat{\emph{Div}}(U,k)$ is an arithmetic adelic divisor with generic fiber $D$ and suppose $\{V_m\}$ is its associated graded linear series. Thne we have
    \[\emph{vol}(V_m)=\emph{vol}(D)\]
\end{lemma}
\begin{proof}
   We begin by choosing a quasi-projective model of $U$. Then clearly we can choose a sequence of model divisors $\overline{D_i}$ which converges to $\overline{D}$ in the boundary topology and satisfies $\overline{D_i}\le \overline{D}$ for each $i$ and let us denote by $V_m^i=H^0(U,mD_i)$. Then clearly there is an inclusion of vector space $V_m^i\xhookrightarrow{} H^0(U,mD)$ for each $i$ and $m$. Moreover the effectivity $\overline{D_i}\le \overline{D}$ implies that for each $s\in V_m^i$, we have 
   \[|s|_{\overline{D},v,m}\le |s|_{\overline{D_i},v,m} \]
   for each $i$ and $m$. Since each $\overline{D_i}$ is a model divisor, it is known from the coherence condition that for any $s\in V_m^i$, we have $||s||_{\overline{D_i},v,m}\le \infty$ for all $v$ and  $||s||_{\overline{D_i},v,m}\le 1$ for almost all $v$. Thus the above inequality implies that $||s||_{\overline{D},v,m}\le \infty$ for all $v$ and $||s||_{\overline{D},v,m}\le 1$ for almost all $v$. As $s\in V_m^i$ was arbitary, we easily deduce that $V_m^i\xhookrightarrow{} V_m$ for all $m$. Thus we can deduce that \[\text{vol}(V_m^i)=\text{vol}(D_i)\le \text{vol}(V_m)\le \text{vol}(D)\]
   for each $i$. However since $\{D_i\}$ converges to $D$ in the boundary topology, we know from \cite[Theorem 5.2.9]{yuan2021adelic} that $\lim_{i\to\infty}\text{vol}(D_i)=\text{vol}(D)$ which clearly concludes the proof.
\end{proof}
We recall the definition of the essential minima of an arithmetic adelic divisor, denoted by $\widehat{\mu}_{\text{ess}}(\cdot)$ as
    \[\widehat{\mu}_{\text{ess}}(\overline{\mathcal{D}}):=\sup_V\inf_{x\in V(\overline{K})}h_{\overline{\mathcal{D}}}(x)\]
    where $V$ ranges over all proper non-empty Zariski open subsets of $U$. Then we can deduce the geometric analogue of \cite[Lemma 5.3.4]{yuan2021adelic}.
\begin{theorem}
    \label{theorem:fundamentalinequalityfunctionfield}
    Suppose $\overline{D}\in\widehat{\emph{Div}}(U,k)$ is an arithmetic adelic divisor with big generic fiber $D$. Then we have
    \[\widehat{\mu}_{\emph{ess}}(\overline{D})\ge \frac{\widehat{\emph{vol}}(\overline{D})}{(d+1)\emph{vol}(D)}\]
\end{theorem}
\begin{proof}
    Suppose we denote by $\overline{V_m}$ the norm family associated to the arithmetic adelic divisor $m\overline{[D}$ which we have shown to be an adelic vector bundle in Lemma \ref{lemma:vectorbundleache}. Furthermore suppose we denote by $\widehat{h^0}(\overline{V_m})$ the $k$-dimension of the space of small sections in $\overline{V_m}$ and by $h^0(V_M)$ the $K$-dimension of the space of all sections $V_m$. Then by \cite[Theorem 4.3.23]{chenmoriwaki} we have
    \begin{equation}
        \label{equa:fundament1}
        \widehat{h^0}(\overline{V_m})-\widehat{\text{deg}}_+(\overline{V_m})\le C\cdot h^0(V_m)
    \end{equation}
    where $\widehat{\text{deg}}_+(\overline{V_m})$ is the positive degree of the adelic vector bundle $\overline{V_m}$ as defined in \cite[Section 4.3.4]{chenmoriwaki}, $C:=\max\{g(C)-1,1\}$ and $g(B)$ is the genus of $B$. Furthermore note that since all the norms $||\cdot||_{\overline{D},v,m}$ are ultrametric, we get by \cite[Prop 4.3.44]{chenmoriwaki} that 
    \[\widehat{\text{deg}}_+(\overline{V_m})=\sum_{i=1}^{h^0(V_m)}\max\{\widehat{\mu_i}(\overline{V_m}),0\}\]
    where $\widehat{\mu_i}(\overline{V_m})$ denotes the $i$-th jumping numbers of the Harder-Narasimhan filtration associated to the adelic vector bundle $\overline{V_m}$. Estimating each of the jumping numbers from above with the largest one in Equation \eqref{equa:fundament1} we get 
    \begin{equation}
        \label{equa:fundament2}
        \widehat{h^0}(\overline{V_m})-h^0(V_m)\cdot \widehat{\mu_1}(\overline{V_m})\le C\cdot h^0(V_m)
    \end{equation}
    Dividing Equation \eqref{equa:fundament2} by $m^{d+1}$ and taking limits, we obtain
    \begin{equation}
        \label{equa:fundament3}
        \frac{\widehat{\text{vol}}(\overline{D})}{(d+1)!}-\frac{\text{vol}(D)}{d!}\cdot\widehat{\mu}^{\text{asy}}_{\text{max}}(\overline{D})\le 0
    \end{equation}
Here we have used the fact that $\lim_{m\to\infty}\frac{\widehat{h^0}(\overline{V_m})}{m^{d+1}}=\frac{\widehat{\text{vol}}(\overline{D})}{(d+1)!}$ which is the very definition of the arithmetic volume and $\lim_{m\to\infty}\frac{h^0(V_m)}{m^d}=\frac{\text{vol}(D)}{d!}$ which is a consequence of Lemma \ref{lemma:ampleser}. Hence we can finally deduce from Equation \eqref{equa:fundament3} that
\begin{equation}
    \label{equa:fundament4}
    \widehat{\mu}^{\text{asy}}_{\text{max}}(\overline{D})\ge \frac{\widehat{\text{vol}}(\overline{D})}{(d+1)\text{vol}(D)}
\end{equation}
Finally we can deduce the claim of the Lemma by noting the inequality $\widehat{\mu}_{\text{ess}}(\overline{D})\ge \widehat{\mu}^{\text{asy}}_{\text{max}}(\overline{D})$. Note that in the projective case this is obtained in \cite[Prop 6.4.4]{chenmoriwaki} and in the quasi-projective case, the arguments can be identically used once we notice that the graded series of normed families $\{\overline{V_m}\}$ are all adelic vector bundles.
\end{proof}
\section{Application to Equidistribution}
In this section, we simultaneously consider the arithmetic and geometric case to obtain a differentiability result for \emph{asymptotic slopes} following the arguments given by Chen for arithmetic divisors on a projective arithmetic variety (see \cite[Proposition 5.1]{chendiff}). As a principal application of it we deduce an equidistribution result for big divisors and we show that it generalises the equidistribution result obtained by Yuan and Zhang in \cite[Theorem 5.4.3]{yuan2021adelic}.\\
For the rest of the section, we work over a quasi-projective variety $U$ over $K$ where $K$ is a number field or a function field of one variable over some field $F$. Note that both the cases are included in the general formalism of valued Dedekind schemes introduced in section 2. Note that then we can view $U$ as an essentially quasi-projective variety over $k=(B,\Sigma)$ in both the cases where $B=\text{Spec}(\mathbb{Z})$ when $K$ is a number field and $B=\mathbb{P}_F^1$ when $K$ is a function field of one variable over $F$ and our considerations in sections 1 and 2 go through in this case. Note that then we have the group of arithmetic adelic divisors on $U$ over $k$ denoted by $\widehat{\text{Div}}(U,k)$ and the group of geometric adelic divisors on $U$ over $K$ which we denote by $\widehat{\text{Div}}(U,K)$. We denote the arithmetic adelic divisors by $\overline{D}$ and the geometric adelic divisors by $D$ for sake of clarity. Furthermore we have the \say{generic fiber} map $\widehat{\text{Div}}(U,k)\rightarrow \widehat{\text{Div}}(U,K)$ as described in \cite[Section 2.5.5]{yuan2021adelic}. Furthermore $d$ will always denote the dimension of $U$ and since we are working assuming that $K$ is either a number field or a function field of one variable this means that any quasi-projective model of $U$ will always have dimension $d+1$.
\begin{remark}
    Suppose $\overline{D}\in\widehat{\text{Div}}(U,k)$ is a big arithmetic adelic divisor. Then from \cite[Theorem 5.2.1(2)]{yuan2021adelic} we can deduce that there is a projective model $\mathcal{X}$ of $U$ and an arithmetic divisor $\overline{D}_1$ such that $\overline{D}_1\le \overline{D}$ in $\widehat{\text{Div}}(U,k)$ and $\overline{D}_1$ is big. Then from \cite[Prop 6.4.18]{chenmoriwaki} we can deduce that the generic fiber $D_1$ of $\overline{D}_1$ is big and $D_1\le D$ in $\widehat{\text{Div}}(U,K)$ where $D$ is the generic fiber of $\overline{D}$. Hence we deduce that the generic fiber $D$ is always big whenever $\overline{D}$ is arithmetically big.   
\end{remark}
For any big arithmetic adelic divisor $\overline{D}\in\widehat{\text{Div}}(U,k)$, we can then define the \say{asymptotic positive slope} $\hat{\mu}_+^{\pi}(\overline{D})$ as
\[\hat{\mu}_+^{\pi}(\overline{D})=\frac{\widehat{\text{vol}}(\overline{D})}{(d+1)\text{vol}(D)}\]
 It is defined in the projective setting by Chen (see \cite[Section 5]{chendiff}) working over number fields.
We derive as a corollary of our differentiability results before that this slope function is differentiable when $\overline{D}$ is big. 
\begin{corollary}
    \label{coroll:asympslope}
    Suppose $\overline{D}$ is a big arithmetic adelic divisor and $\overline{M}$ an integrable arithmetic adelic divisor in $\widehat{\emph{Div}}(U,k)$. Then we have
     \[\lim_{t\to 0}\frac{\hat{\mu}_+^{\pi}(\overline{D}+t\overline{M})-\hat{\mu}_+^{\pi}(\overline{D})}{t}=\frac{\langle\overline{D}^d\rangle\cdot\overline{M}}{\emph{vol}(D)}-\frac{d\langle D^{d-1}\rangle\cdot M}{\emph{vol}(D)}\hat{\mu}_+^{\pi}(\overline{D})\]
\end{corollary}
\begin{proof}
    The proof is quite straightforward by applying the following elementary principle for Analysis 1. If $f(t),g(t)$ are two differentiable functions on $\mathbb{R}$, then the quotient $\frac{f(t)}{g(t)}$ is also differentiable at a point $t_0$ such that $g(t_0)\neq 0$ and the derivative at $t_0$ is given by $\frac{g(t_0)f'(t_0)-f(t_0)g'(t_0)}{g(t_0)^2}$. We apply the above to the functions $f(t)=\widehat{\text{vol}}(\overline{D}+t\overline{M})$ and $g(t)=\widehat{\text{vol}}(D+tM)$ which are both differentiable by Theorem \ref{theorem:finaldiff} to deduce the claim since $\mu_+^{\pi}(\overline{D}+t\overline{M})=\frac{1}{(d+1)}\frac{f(t)}{g(t)}$.
\end{proof}
\begin{remark}
    Note that there is a notion of \say{vector valued heights} as introduced in \cite[Section 5.3.1]{yuan2021adelic} and in the case of a number field or function field in one variable, by composing with the Arakelov degree on $\widehat{\text{Pic}}(K)$ we get the notion of a real valued height. When the arithmetic adelic divisor $\overline{D}$ is also nef in addition to satisfying the hypotheses of Corollary \ref{coroll:asympslope} then it can be checked easily using \cite[Theorem 5.2.2(1)]{yuan2021adelic} that the above asymptotic positive slope is exactly equal to $h_{\overline{D}}(U)$. However it might not be true that the perturbations $\overline{D}+t\overline{M}$ are still nef for very small $t$ and hence we cannot deduce from Corollary that the height functions from \cite{yuan2021adelic} are differentiable at big points.
\end{remark}
We want to use the differentiability of the asymptotic positive slope to study the well known \say{variational principle}. We fix a place $v$ over $K$ and we denote by $U_v$ the $K_v$-analytic space associated to $U\times_K K_v$ where $K_v$ is the completion of $K$ at $v$. Note that as explained in \cite[Chapter 1]{berkovich}, they are the the usual complex analytic spaces associated to $U\times_K K_v$ when $v$ is complex Archimedean and the quotient of the usual complex analytic spaces of $U\times_K K_v$ by complex conjugation when $v$ is real Archimedean and the Berkovich spaces studied in \cite{berkovich} when $v$ is non-Archimedean. For a geometric point $x\in U(\overline{K})$, we consider the discrete measure \[\eta_x:=\frac{1}{\#O(\tilde{x})}\sum_{y\in O(\tilde{x})}\delta_y\]
on $U_v$ where $\tilde{x}$ denotes the induced closed point on $U\times_K \overline{K_v}$ from $x$, $O(\tilde{x})$ denotes the Galois orbit of $x$ under the action of $\text{Gal}(\overline{K_v}/K_v)$ and $\delta_y$ is the Dirac measure concentrated at $y$.\\
Now we consider a sequence of points $\{x_m\}$ of geometric points in $U(\overline{K})$ and we recall two definitions of sequence of points crucial in stating equidistribution conjectures. Recall that given an arithmetic adelic divisor $\overline{D}$ on $\mathcal{U}$, there is real valued height function $h_{\overline{D}}(\cdot)\colon U(\overline{K})\rightarrow \mathbb{R}$ as explained in \cite[Section 5.3.1]{yuan2021adelic}.
\begin{definition}
    We call a sequence of geometric points $\{x_m\}$ in $U(\overline{K})$ \emph{generic} if the sequence is dense in the Zariski topology of $U$ \emph{i.e} for all non-empty Zariski open subset $V$ of $U$, $x_m\in V(\overline{K})$ for large enough $m$. Given an arithmetic adelic divisor $\overline{D}$, we call $\{x_m\}$ is \emph{small} w.r.t $\overline{D}$ if $\lim_{m\to\infty}h_{\overline{D}}(x_m)=\mu_+^{\pi}(\overline{D})$
\end{definition}
\begin{remark}      
    Note that the definition of a small sequence above is slightly different from the usual definition of a small sequence in literature (or \emph{directionally small} as termed by Yuan and Zhang, see \cite[Section 5.4.1]{yuan2021adelic}). However when $\overline{D}$ is nef then our definition agrees withe classical cone since then $\mu_+^{\pi}(\overline{D})=h_{\overline{D}}(U)$. 
\end{remark}
 The equidistribution theorems and conjectures deal with the problem of determining conditions on the sequence $\{x_m\}$ and $\overline{D}$ such that the sequence of integrals $\int_{U_v} g\ d\eta_{x_m}$ converges to some real number in $\mathbb{R}$ as $m\to\infty$ for every continuous function $g$ on $U_v$ with compact support.\\
 In \cite[Theorem 3.6.4]{yuan2021adelic}, Yuan and Zhang introduce the notion of \say{compactified Green functions} as
 \[C^0_{\text{cptf}}(U_v):=\{f\in C^0(U_v)\mid \frac{f}{g_0}\to 0\ \text{along the boundary}\}\]
 where $g_0$ is the Green function of any boundary divisor. Theorem \cite[Theorem 3.6.4]{yuan2021adelic} tells us that for any $g\in C^0_{\text{cptf}}(U_v)$ there is an adelic divisor $\mathcal{O}(g)$ in  $\widehat{\text{Div}}(U,k)$. Furthermore construction identifies $C^0_{\text{cptf}}(U_v)$ with a subgroup of $\widehat{\text{Div}}(U,k)$ and contains the subgroup of compactly supported continuous functions $C_c^0(U_v)$ on $U_v$. The following description of the height function of $\mathcal{O}(g)$ follows easily by construction of intersection numbers when $\mathcal{O}(g)$ is integrable and allows us to relate the measures $\eta_{x_m}$ with the integrals $\int_{U_v} g\ d\eta_{x_m}$  
 \begin{equation}
     \label{eqn:heightdescription}
      h_{\mathcal{O}(g)}(x_m)=\int_{U_v} g\ d\eta_{x_m}
 \end{equation}
 We want use the variational principle and obtain our equidistribution from differentiability in the direction of the divisor $\mathcal{O}(g)$ but $\mathcal{O}(g)$ is not integrable in the sense of Yuan and Zhang for any arbitary $g$. However there are a particular class of functions for which the associated divisor is integrable and they approximate any arbitary function with compact support under the uniform topology which we introduce next.
 \begin{definition}
     \label{def:model}
     Suppose $g\in C^0(U_v)$ is a continuous function on $U_v$. We say that $g$ is a \emph{model function} in either of the two following cases
     \begin{enumerate}
         \item $v$ is Archimedean and $g$ is the restriction of a smooth function on some projective model of $U_v$.
         \item $v$ is non-Archimedean and there is a projective model $\mathcal{X}$ of $U\times_B \emph{Spec}(O_{B,v})$  over the local ring $O_{B,v}$ of $B$ at $v$ and a vertical Cartier divisor $\mathcal{A}$ on $\mathcal{X}$ such that a positive multiple of $g$ is the function induced by $\mathcal{A}$ on $U_v$ via restriction.
     \end{enumerate}
     We denote the group of model functions on $U_v$ by $C_{\emph{mod}}(U_v)$.
 \end{definition}
 Note that any model function $f$ is bounded and hence in particular $C^0_{\text{mod}}(U_v)\subseteq C^0_{\text{cptf}}(U_v)$. We go on to show that the adelic divisors $\mathcal{O}(g)$ are integrable in the sense of Yuan and Zhang when $g$ is model.
 \begin{lemma}
     \label{lemma:modelintegrable}
     Suppose $g\in C_c^0(U_v)$ is a smooth model function with compact support. Then the associated arithmetic adelic divisor $\mathcal{O}(g)$ is integrable.
 \end{lemma}
 \begin{proof}
     We first assume the case when $v$ is non-archimedean. Then since $g$ is model, be definition there is a projective model $\mathcal{X}$ of $U$ over $k$ with generic fiber $X$ and a vertical divisor $\mathcal{A}$ on $\mathcal{X}$ such that $g$ is induced by the model function of $\mathcal{A}$. Here we identify $g$ as a continuous function on $X_v$ and $U_v$ by extending it to 0 outside its support. Then using Serre's theorem we can write $\mathcal{A}=\mathcal{A}_1-\mathcal{A}_2$ where $\mathcal{A}_1$ and $\mathcal{A}_2$ are ample divisors on $\mathcal{X}$. Then clearly writing $\mathcal{O}(g)=\mathcal{O}(g_1)-\mathcal{O}(g_2)$ where $g_1$ and $g_2$ denotes the induced model functions from $\mathcal{A}_1$ and $\mathcal{A}_2$ respectively, we conclude that $\mathcal{O}(g)$ is integrable.\\
     For Archimedean $g$, we give a sketch of the argument. Like before we choose a projective model $\mathcal{X}$ of $U$ with generic fiber $X$. Since $g$ is smooth, $\mathcal{O}(g)$ induces a smooth Chern form on $X_v$. Choose an ample divisor $A$ on $X$ with usual Fubiny-Study form whic h has positive curvature and adding high enough powers of it to $\mathcal{O}(g)$, we get that $\mathcal{O}(g)+mA$ also has positive curvature. Then writing $\mathcal{O}(g)=\mathcal{O}(g)+mA-mA$ we deduce the claim. We refer to \cite[Prop 10.4]{localcansaviour} and \cite[Section 2.1.1]{yuan2021adelic} for more general versions and details.
 \end{proof}
 \begin{remark}
     To highlight the importance of the model functions,  we remark here that these functions approximate arbitary continuous functions on $U_v$ with compact support (looking at them as continuous functions on the analytification of some projective model of $U$) under the uniform topology. Indeed this is the famous Weierstrass approximation theorem in the Archimedean case and was proved by Gubler in \cite[Theorem 7.12]{saviourpiecewiselinear} in the non-archimedean case over algebraically closed fields. Gubler's arguments were adapted by Yuan \cite[Section 3]{bigyuan} for the non-algebraically closed case. In particular this shows that the functions considered in Lemma \ref{lemma:modelintegrable} approximate arbitary continuous functions uniformly in both the cases of number fields and function fields and this fact will be important to us to reduce the case of weak convergence to model situations. Furthermore note that the space of model functions with compact support form a subgroup of $\widehat{\text{Div}}(U,k)$ under the inclusion into $C^0_{\text{cptf}}(U_v)$ since model functions are always bounded and we denote this subgroup of $\widehat{\text{Div}}(U,k)$ by $C^0_{\text{mod}}(U_v)$.
 \end{remark}
 Our strategy to attack the equidistribution problem is identical to the strategy in \cite[Section 5.1]{chendiff} but we still repeat the arguments for clarity. Given a sequence of geometric points $\{x\}:=\{x_m\}_{m\in\mathbb{N}}$ in $U(\overline{K})$, we begin by defining the following function 
 \[\phi_{\{x\}}\colon \widehat{\text{Div}}(U,k)\rightarrow \mathbb{R}\cup \{\pm \infty\}\]
 \[\overline{D}\mapsto \liminf_{m\to\infty} h_{\overline{D}}(x_m)\]
 The additivity of $h_{\overline{D}}(\cdot)$ on  geometric points easily shows that $\phi_{\{x\}}(\cdot)$ is a super-additive function on $\widehat{\text{Div}}(U,k)$. Next we record two important properties of $\phi_{\{x\}}(\cdot)$ that will be relevant in our future considerations. Recall that a sequence of Radon measures $\{d\eta_m\}$ on $U_v$ is said to \emph{weakly converge} to another Radon measure $d\eta$ if $\lim_{m\to\infty}\int_{U_v}g\ d\eta_m=\int_{U_v}g\ d\eta$ for all $g\in C_c^0(U_v)$. If the limit measure is not specified and the integrals converge to some real number, then we say the sequence $\{d\eta_m\}$ converges weakly.
 Given a big arithmetic adelic divisor $\overline{D}$, we can define the sub-semigroup
 \[C(U, \overline{D}):=\{m\overline{D}+\mathcal{O}(f)\mid m\in\mathbb{N}, \ m\overline{D}+\mathcal{O}(f)\ \text{is big and}\ f\in C^0_{\text{mod}}(U_v)\}\]
 of $\widehat{\text{Div}}(\mathcal{U},k)$.\\
 Recall that given a sub-semigroup $C$ and a subgroup $H$ of an abelian group $G$, we say that $C$ is \emph{open with respect to} $H$ if for all $g\in C$ and $h\in H$, we have that $mg+h\in C$ for all large enough positive integers $m$. It is easy to see that $C(U,\overline{D})$ is open with respect to the sub-group $C^0_{\text{mod}}(U_v)$ which will be crucial to use \cite[Proposition 5.4]{chendiff}. Indeed given $g\in C^0_{\text{mod}}(U_v)$ and $f\in C^0_{\text{mod}}(U_v),\ m\in\mathbb{N}$ such that $m\overline{D}+\mathcal{O}(f)\in C(U, \overline{D})$ , for large enough positive integer $n$ we have $n(m\overline{D}+\mathcal{O}(f))+\mathcal{O}(g)=mn\overline{D}+\mathcal{O}(nf+g)\in C(U, \overline{D})$ since $nf+g\in C^0_{\text{mod}}(U_v)$ for all positive integers $n$, the adelic volume function $\widehat{\text{vol}}(\cdot)$ is positive homogeneous and continuous and $m\overline{D}+\mathcal{O}(f)$ itself is big. Next we record two lemmas which will be useful later.
 \begin{lemma}
     \label{lemma:derlink}
       Suppose $\overline{D}\in\widehat{\emph{Div}}(U,k)$ and $\{x\}:=\{x_m\}$ be a sequence of geometric points in $U(\overline{K})$ such that the sequence $\{h_{\overline{D}}(x_m)\}$ is convergent. For any $f\in C^0_{\emph{cptf}}(U_v)$, the following are equivalent
       \begin{itemize}
           \item The sequence $\{\int_{U_v}f\ d\eta_{x_m})\}$ converges as $m\to\infty$.
           \item The function $\phi_{\{x\}}(\cdot)$ is differentiable at $\overline{D}$ in the direction $\mathcal{O}(f)$.
       \end{itemize}
       Furthermore if one of the conditions hold, then
        \[\lim_{m\to\infty}\int_{U_v}f\ d\eta_{x_m}=\phi_{\{x\},{\overline{D}}}'(f)\]
         where $\phi_{\{x\},{\overline{D}}}'(f)$ denotes the differential of $\phi_{\{x\}}$ at $\overline{D}$ calculated at $\mathcal{O}(f)$.
 \end{lemma}
 \begin{proof}
     The proof of the equivalence is identical to the proof of \cite[Theorem 5.3]{chendiff} but we still sketch a proof for clarity. The series of equalities that is crucial for the proof is
     \begin{equation}
         \label{equa:withoutme}
         \phi_{\{x\}}(n\overline{D}+\mathcal{O}(f))=\liminf_{m\to\infty}(nh_{\overline{D}}(x_m)+\int_{U_v}f\ d\eta_{x_m})=n\phi_{\{x\}}(\overline{D})+\phi_{\{x\}}(\mathcal{O}(f))
     \end{equation}
     for all positive integers $n$ and $f\in C^0_{\text{cptf}}(U_v)$.
    The first equality above easily follows from the additivity of the height function on the geometric points. Now suppose that the first assertion holds. Then the second equality easily follows since the second summand in the middle term of equality \eqref{equa:withoutme} is convergent as $m\to\infty$. Replacing $f$ by $-f$ in the above inequality clearly shows that the function $\phi_{\{x\}}$ is differentiable at $\overline{D}$ along the direction of $f$. In fact it shows that it is linear in the direction of $f$.\\
    Conversely if we assume that the second assertion holds, then also the second equality of the equation \eqref{equa:withoutme} holds but this time since the \textbf{first} summand of the middle term in equation \eqref{equa:withoutme} is convergent.  The differentiability of $\phi_{\{x\}}$ then allows us to deduce by taking limits as $n\to\infty$ that 
    \[\phi_{\{x\},\overline{D}}'(\mathcal{O}(f))=\phi_{\{x\}}(\mathcal{O}(f))\]
    where the left hand-side above denotes the differential of $\phi_{\{x\}}$ at $\overline{D}$ in the direction of $\mathcal{O}(f)$. Since differentials are linear maps, replacing $\mathcal{O}(f)$ by $\mathcal{O}(-f)$ we deduce from the above equality that 
    \[\phi_{\{x\}}(\mathcal{O}(-f))=-\phi_{\{x\}}(\mathcal{O}(f))\]
    Recalling the definition of the function $\phi_{\{x\}}$ we deduce that \[\limsup_{m\to\infty}\int_{U_v}f\ d\eta_{x_m}=\liminf_{m\to\infty}\int_{U_v}f\ d\eta_{x_m}= \phi'_{\{x\},\overline{D}}(\mathcal{O}(f))\]
    which proves the first assumption. Note that equation \eqref{equa:withoutme} also shows the last additional claim.
 \end{proof}
  \begin{lemma}
     \label{lemma:derlink1}
     If $\{x\}$ is a generic sequence in $U(\overline{K})$, then we have 
         \[\phi_{\{x\}}(\overline{D})\ge \mu_+^{\pi}(\overline{D})\]
         for big arithmetic adelic divisors $\overline{D}$ in $\widehat{\emph{Div}}(U,k)$.
 \end{lemma}
 \begin{proof}
    We recall the definition of the essential minima of an arithmetic adelic divisor, denoted by $\widehat{\mu}_{\text{ess}}(\cdot)$ as
    \[\widehat{\mu}_{\text{ess}}(\overline{D}):=\sup_V\inf_{x\in V(\overline{K})}h_{\overline{D}}(x)\]
    where $V$ ranges over all proper non-empty Zariski open subsets of $U$. Then the definition of $\phi_{\{x\}}$ readily shows that for any Zariski open $V\subset U$, since $x_m\in V(\overline{K})$ for large enough $m$ we have that $\phi_{\{x\}}\ge\widehat{\mu}_{\text{ess}}$ on $\widehat{\text{Div}}(U,k)$. Then we can conclude the claim of the lemma since $\widehat{\mu}_{\text{ess}}\ge\mu^{\pi}_+$ from \cite[Lemma 5.3.4]{yuan2021adelic}.
 \end{proof}
 We are going to use the above Lemma together with our differentiability result to obtain an equidstribution theorem. As the above lemma suggests our strategy will be to deduce the  differentiability of $\phi_{\{x\}}$ at big adelic divisors from the differentiability of $\pi_+^{\pi}$. We will need a result for that on differentiability of super-additive functions as done in \cite{chendiff} which we restate for the sake of completeness.
 \begin{lemma}{\cite[Prop 5.4]{chendiff}}
     \label{theorem:auxillchen}
     Let $G$ be a group, $H$ be a sub-group, $C$ be a sub semi-group which is open with respect to $H$ and $x\in C$. Furthermore let $f,t\colon C\rightarrow\mathbb{R}$ be two positively homogeneous functions satisfying
     \begin{enumerate}
         \item For all $a,b\in C$, we have $f(a+b)\ge f(a)+f(b)$
         \item $f\ge t$, $f(x)=t(x)$
         \item $t$ is differentiable at $x$ along the directions in $H$.
     \end{enumerate}
Then the function $f$ is differentiable as well and we have $D_x f=D_x t$.
 \end{lemma}
 Next we show that the positive intersections respect effectivity relations in terms of the direction $\overline{E}$ which helps us to control positive intersection products against generically trivial divisors with constant metrics.
 \begin{lemma}
     \label{lemma:control}
     Suppose $\overline{D}$ is an arithmetic adelic divisor in $\widehat{\emph{Div}}(U,k)$ with generic fiber $D$ and suppose $\overline{E}_1, \overline{E}_2\in \widehat{\emph{Div}}(U,k)_{\emph{int}}$ such that $\overline{E}_1\le \overline{E}_2$. Then we have
     \[\langle\overline{D}^d\rangle\cdot \overline{E}_1\le \langle\overline{D}^d\rangle\cdot \overline{E}_2\]
     Furthermore for any $r>0\in \mathbb{R}$, we have 
     \[\langle\overline{D}^d\rangle\cdot \mathcal{O}(r)\le r\cdot\emph{vol}(D)\]
 \end{lemma}
 \begin{proof}
     We begin with proof of the first statement. As usual we choose a common quasi-projective model of $U$ and suppose $(X,\overline{A})$ is an admissible approximation and suppose $A$ is the generic fiber. Then since $\overline{E}_2-\overline{E}_1\ge 0$ and $\overline{A}$ is nef, we have
     \[\overline{A}^d\cdot(\overline{E}_2-\overline{E}_1)\ge 0\Rightarrow \overline{A}^d\cdot \overline{E}_1\le \overline{A}^d\cdot \overline{E}_2\]
     Since the positive intersection products are defined by taking the supremum of all such products appearing as we vary $(X',\overline{A})$ across all admissible approximations in the above display we can deduce the claim.\\
     For the second claim, we note that $\mathcal{O}(r)$ is nef and hence in particular integrable for all constants $r>0$ and so it makes sense to take positive intersection products against it. As in the first paragraph we take an admissible approximation $(X',\overline{A})$ of $\overline{D}$ on a common quasi-projective model. Then we have 
     \begin{equation}
         \label{equa:nebula}
         \overline{A}^d\cdot \mathcal{O}(r)=r\cdot A^d\le r\cdot \text{vol}(D)\ \text{since}\ r\ge 0
     \end{equation}
     To see the last inequality above, we note that $A\le D$ in the level of generic fibers and hence $\text{vol}(A)\le\text{vol}(D)$. However since $\overline{A}$ was arithmetically nef (as it was an admissible approximation of $\overline{D}$), we deduce that $A$ is nef and hence $\text{vol}(A)=A^d\le \text{vol}(D)$ and hence we deduce the inequality in \eqref{equa:nebula}. Finally to obtain our second claim, we note that positive intersection product $\langle\overline{D}^d\rangle\cdot \mathcal{O}(r)$ is defined by taking the supremum of the first term in \eqref{equa:nebula} as we vary admissible approximations $(X',\overline{A})$ which gives us the desired inequality.
 \end{proof}
 Next we extend the definition of positive intersection products against generically trivial adelic line bundles $\mathcal{O}(g)$ where $g\in C_{\text{cptf}}(U_v)$ is arbitary. Using the lemma above, we show that we can extend positive intersection products to uniform limits of model functions. Recall from Definition \ref{def:model} that we say $g\in C_c^0(U_v)$ is a \emph{model function} if $g$ is smooth when $v$ is Archimedean and if $g$ is induced by a model when $v$ is non-Archimedean.
 \begin{lemma}
     \label{lemma:positextension}
     Suppose $g_n$ is a sequence of uniformly convergent model functions in $C_{\emph{mod}}(U_v)$ converging uniformly to an arbitary continuous function $g\in C_{\emph{cptf}}(U_v)$ and $\overline{D}\in\widehat{\emph{Div}}(U,k)$ is big. Then we have that the sequence \[\{\langle\overline{D}^d\rangle\cdot \mathcal{O}(g_n)\}\]
     converges in $\mathbb{R}$. Furthermore the limit is independent of the sequence of functions $g_n$.
 \end{lemma}
 \begin{proof}
     We only prove the first assertion as the second one can be deduced very similarly. As we have remarked before, since the functions $g_n$ are model we know that each $\mathcal{O}(g_n)$ is integrable and hence the terms $\langle\overline{D}^d\rangle\cdot \mathcal{O}(g_n)$ make sense. Since we have that $\{g_m\}$ is uniformly convergent, we can find a sequence of positive rational numbers $r_m$ converging to 0 such that $-r_m\le g_n-g_m\le r_m$ for all $n\ge m$. In terms of Hermitian line bundles, we have then the effectivity relations 
     \begin{equation}
         \label{equa:nebula2}
         -\mathcal{O}(r_m)\le \mathcal{O}(g_n)-\mathcal{O}(g_m)\le \mathcal{O}(r_m)
     \end{equation}
     Using the linearity of positive intersection products in Lemma \ref{lemma:positlinear} and the first assertion of Lemma \ref{lemma:control}, we deduce 
     \[-\langle\overline{D}^d\rangle\cdot \mathcal{O}(r_m)\le \langle\overline{D}^d\rangle\cdot\mathcal{O}(g_n)-\langle\overline{D}^d\rangle\cdot\mathcal{O}(g_m)\le \langle\overline{D}^d\rangle\cdot \mathcal{O}(r_m)\]
     Estimating the two extremities of the above inequality using the second assertion of Lemma \ref{lemma:control}, we get 
     \[-r_m\cdot\text{vol}(D)\le\langle\overline{D}^d\rangle\cdot\mathcal{O}(g_n)-\langle\overline{D}^d\rangle\cdot\mathcal{O}(g_m)\le r_m\cdot\text{vol}(D)\ \text{for all}\ n\ge m\]
     This gives us the desired claim since $r_m\to 0$ as $m\to\infty$.
 \end{proof}
 We can then finally extend our positive intersection products against $\mathcal{O}(g)$ where $g$ is an arbitary continuous function with compact support.
 \begin{definition}
     \label{def:positagainstconti}
     Suppose $\overline{D}$ is any big arithmetic adelic divisor in $\widehat{\emph{Div}}(U,k)$ and suppose $g\in C_c^0(U_v)$. Then we define 
     \[\langle\overline{D}^d\rangle\cdot \mathcal{O}(g):=\lim_{m\to\infty}\langle\overline{D}^d\rangle\cdot \mathcal{O}(g_m)\]
     where $\{g_m\}$ is any sequence of model functions on $U_v$ converging uniformly to $g$.
 \end{definition}
 \begin{remark}
     We remark that our definition above makes sense. Indeed choosing a projective model $X$ of $U$ and viewing $g$ as a continuous function on $X_v$ by extending it to 0 outside the support, we can use the Weierstrass approximation theorem in the Archimedean case and \cite[Theorem 7.12]{saviourpiecewiselinear} in the algebraically closed non-Archimedean case (adapted by Yuan in \cite{bigyuan} in the arbitary non-Archimedean case) to deduce that there is indeed such a sequence $\{g_m\}$ approximating $g$ uniformly as above. Furthermore in the second assertion of Lemma \ref{lemma:positextension} we also see that the limit is independent of the sequence $\{g_m\}$.
 \end{remark}
 We can finally state our main result on equidistribution of small and generic points with respect to a big adelic divisor.
 \begin{theorem}
     \label{theorem:equidistribution}
     Suppose $U$ is a quasi-projective variety over $K$ of dimension $d§$ and suppose $\overline{D}$ is a big arithmetic adelic divisor in $\widehat{\emph{Div}}(U,k)$. Furthermore suppose $\{x_m\}$ is a generic sequence of geometric points in $U(\overline{K})$ which is small with respect to $\overline{D}$. Then for any place $v$ on $K$ and for any $g\in C_c^0(U_v)$, we have
     \[\lim_{m\to\infty}\int_{U_v}g\ d\eta_{x_m}=\frac{\langle\overline{D}^d\rangle\cdot\mathcal{O}(g)}{\emph{vol}(D)}\]
     In particular, the sequence of Radon measures $\{\eta_{x_m}\}$ converge weakly to the Radon measure given by 
     \[g\in C_c^0(U_v)\mapsto \frac{\langle\overline{D}^d\rangle\cdot\mathcal{O}(g)}{\emph{vol}(D)}\]
 \end{theorem}
 \begin{proof}
      We begin by noting that we can reduce to the case when $g$ is a model function. Indeed as we have remarked before, any continuous function with compact support on $U_v$ can be approximated by model functions in the uniform topology. Note that then for such a sequence of model functions approximating our arbitary continuous function, the L.H.S will converge to the desired integral and the R.H.S will also converge to the desired quantity simple by how we define the positive intersection products against $\mathcal{O}(g)$ (see Definition \ref{def:positagainstconti} and Lemma \ref{lemma:positextension}). Hence we can safely assume that $g$ is a model function.\\
     Note that now since we have assumed $g$ to be model, from Lemma \ref{lemma:modelintegrable} we know that each $\mathcal{O}(g)\in \widehat{\text{Div}}(U,k)_{\text{int}}$. The proof is then a consequence of all the results that we obtained in this section so far. We begin by considering Theorem \ref{theorem:auxillchen} and suppose we set $G=\widehat{\text{Div}}(U,k)$, $H=C^0_{\text{mod}}(U_v)$, $C=C(U,\overline{D})$. We set $f=\phi_{\{x\}}$ and $t=\mu_+^{\pi}$. We check that these functions satisfy the hypotheses of Theorem \ref{theorem:auxillchen}. Indeed note that we get from positive homogeneity of heights and volumes that both $f$ and $t$ are positively homogeneous functions. Moreover
     \[\phi_{\{x\}}(\overline{D}_1+\overline{D}_2)=\liminf_{m\to\infty}h_{\overline{D}_1+\overline{D}_2}(x_m)=\liminf_{m\to\infty}(h_{\overline{D}_1}(x_m)+h_{\overline{D}_2}(x_m))\ge \phi_{\{x\}}(\overline{D}_1)+\phi_{\{x\}}(\overline{D}_2)\]
     so $f=\mu_+^{\pi}$ satisfies super-additivity.
     From Lemma \ref{lemma:derlink1} we have that $f\ge t$ on $C(U,\overline{D})$ since all its elements are big by definition. Furthermore since the sequence $\{x\}$ is small with respect to $\overline{D}$, from the definitions of $\phi_{\{x\}}$ and $\mu_+^{\pi}$ we have that $f(x)=t(x)$ and hence also the second hypothesis of Theorem $\ref{theorem:auxillchen}$ is satisfied. Finally from Corollary \ref{coroll:asympslope} and the assumed bigness of $\overline{D}$, we deduce that $t$ is differentiable at $x$ in the direction of $\mathcal{O}(g)$ since $\mathcal{O}(g)$ is integrable as mentioned in the first paragraph. Furthermore we also deduce from Corollary \ref{coroll:asympslope} that the differential is given by \[D_x t(g)=\frac{\langle\overline{D}^d\rangle\cdot\mathcal{O}(g)}{\text{vol}(D)}\]
     since the underlying generic fiber of $\mathcal{O}(g)$ is trivial. Hence from Theorem \ref{theorem:auxillchen} we deduce that $f=\phi_{\{x\}}$ is differentiable at $x=\overline{D}$ with the differential given by the differential of $t$ at $x$. Then from Lemma \ref{lemma:derlink} we deduce that the sequence of Radon measures $\{\eta_{x_m}\}$ converges weakly to the Radon measure 
     \[g\in C^0_c(U_v)\mapsto\lim_{m\to\infty}\int_{U_v}g\ d\eta_{x_m}=D_x f(g)=D_x t(g)=\frac{\langle\overline{D}^d\rangle\cdot\mathcal{O}(g)}{\text{vol}(D)}\]
     which is what we wanted to prove.
 \end{proof}
 \section{Relation to other Equidistribution Results}
 In this section, we relate our equidistribution result obtained in Theorem \ref{theorem:equidistribution} to two other equidistribution results already known. We recall that we work over a quasi-projective variety $U$ over $K$ where $K$ is either a number field or a function field of a smooth projective curve over a field. We denote by $k$ the base valued scheme $(B,\Sigma)$ where $B$ is the ring of integers of $K$ when $K$ is a number field or $B$ is the unique smooth projective curve over a field whoose function field is $K$ and $\Sigma$ is empty. The first equidistribution result is due to Yuan and Zhang obtained for nef adelic line bundles on quasi-projective varieties as in \cite[Theorem 5.4.3]{yuan2021adelic} and we show it can be deduced by our result. We end by remarking the relation of our result with a previously obtained equidistribution of Berman and Boucksom\\
 Recall that given an integrable arithmetic adelic divisor $\overline{D}\in\widehat{\text{Div}}(U,k)_{\text{int}}$, there is a Radon measure on $U_v$, denoted as the \emph{Chambert-Loir measure} by Yuan and Zhang, corresponding to $\overline{D}$ as explained in \cite[Section 3.6.6]{yuan2021adelic}. We denote this Radon measure by $c_1(\overline{D})^d$. They are constructed as weak limits of the already existing measures for the projective models approximating $\overline{D}$. The property crucial for us is that when $\overline{D}$ is nef, we have for any $g\in C_c^0(U_v)$ the equality
 \begin{equation}
     \label{equa:crucial}
     \overline{D}^d\cdot \mathcal{O}(g)=\int_{U_v}g\ c_1(\overline{D})^d
 \end{equation}
 Indeed note that a similar inequality holds in the projective setting by the definition of arithmetic intersections (see \cite{GilletSoule}). Then the above equality can be deduced by noting that both the left and right hand sides are constructed by taking limits along projective models approximating $\overline{D}$ as in \cite[Prop 4.1.1]{yuan2021adelic} and \cite[section 3.6.6]{yuan2021adelic} respectively.\ Then we can deduce Yuan and Zhang's equidistribution in the next theorem.
\begin{theorem}{\cite[Theorem 5.4.3]{yuan2021adelic}}
\label{theorem:deduction1}
Suppose $\overline{D}\in\widehat{\emph{Div}}(U,k)_{\emph{nef}}$ is a nef arithmetic adelic divisor such that $\emph{vol}(D)=D^d>0$ and suppose $\{x_m\}$ is a generic sequence in $U(\overline{K})$ which is small with respect to $\overline{D}$. Then the sequence of Radon measure $\{d\eta_{x_m}\}$ converge weakly to the Radon measure $\frac{c_1(\overline{D})^d}{D^d}$.
\end{theorem}
\begin{proof}
    We have to show that for all $g\in C_c^0(U_v)$, we have 
    \[\lim_{m\to\infty}\int_{U_v} g\ d\eta_{x_m}\mapsto \frac{1}{D^d}\cdot\int_{U_v} g\ c_1(\overline{D})^d\]
    The trick is to deduce it from Theorem \ref{theorem:equidistribution} by twisting $\overline{D}$. Note that for any $r>0$, we have that the $r$-twist $\overline{D}(r)$ is again nef. Moreover since it is nef, we have $\widehat{\text{vol}}(\overline{D}(r))=\overline{D}(r)^{d+1}=\overline{D}^{d+1}+(d+1)rD^d>0$ from Theorem \cite[Theorem 5.2.2(1)]{yuan2021adelic} and since $D^d>0$ by hypothesis. Hence we deduce that the $r$-twist $\overline{D}(r)$ is big for all $r>0$. From the previous equality we can also deduce $\widehat{\text{vol}}(\overline{D}(r))=\widehat{\text{vol}}(\overline{D})+(d+1)r\cdot \text{vol}(D)\Rightarrow \hat{\mu}^{\pi}_+(\overline{D}(r))=\hat{\mu}^{\pi}_+(\overline{D})+r$. Moreover we have $h_{\overline{D}(r)}(x_m)=h_{\overline{D}}(x_m)+r$ and hence we deduce that the sequence $\{x_m\}$ is small with respect to $\overline{D}(r)$ as well. Then we can apply Theorem \ref{theorem:equidistribution} and deduce that the sequence of measures $\{d\eta_{x_m}\}$ converges weakly to the measure given by 
    \[g\in C_c^0(U_v)\mapsto \frac{\langle\overline{D}(r)^d\rangle\cdot \mathcal{O}(g)}{D^d}\]
    However note that since $\overline{D}(r)$ is already nef, by Corollary \ref{corol:positiisusual} we have that in this case the positive intersection products are the same as usual intersection products and hence
    \[\frac{\langle\overline{D}(r)^d\rangle\cdot \mathcal{O}(g)}{D^d}=\frac{\overline{D}(r)^d\cdot \mathcal{O}(g)}{D^d}\]
    However note that since $\mathcal{O}(g)$ is generically trivial and $\mathcal{O}(r)$ is also generically trivial with constant metric, we have that $\overline{D}(r)^d\cdot \mathcal{O}(g)=\overline{D}^d\cdot \mathcal{O}(g)$ which finishes the proof together with Equation \eqref{equa:crucial}.

\end{proof}
\begin{remark}
    We end this section by relating our equidistribution with an equidistribution result obtained by Berman and Boucksom. In \cite[Theorem D]{BermanBoucksom} Berman and Boucksom essentially show that for a smooth projective variety over $\mathbb{C}$ and a big line bundle on it endowed with a continuous metric denoted by $\overline{L}$, the Galois orbits of a generic sequence $\{x_m\}$ equidistribute to the \emph{equilibrium measure} whenever the height of the points with respect to $\overline{L}$ converge to the \emph{adelic energy at equilibrium}. They obtain it as a corollary of differentiability of \emph{energy at equilibrium} as in \cite[Theorem A]{BermanBoucksom}. Our equidistribution can be thought of as a generalisation of Berman-Boucksom equidistribution in the quasi-projective setting. However as in Chen's equidistribution (see \cite[Corollary 5.5]{chendiff}) we need an additional assumption of the divisor $\overline{D}$ to be arithmetically big compared to the more relaxed positivity assumptions of Berman-Boucksom.
\end{remark}

\input{contact}

\printbibliography

\end{document}

%% file: classific.tex
\section*{Keywords:}
\textbf{2020 Mathematics Subject Classification :\ 14G40}

%% file: contact.tex
\section*{Contact}
Address:\ \emph{Debam Biswas,\ Department of Mathematics,\ University of Regensburg\\
Universitatstrasse 31, 93053, Regensburg}\\
Email:\ \href{mailto:debambiswas@gmail.com}{debambiswas@gmail.com}

%% file: ref.bib
@misc{yuan2021adelic,
      title={Adelic line bundles over quasi-projective varieties}, 
      author={Xinyi Yuan and Shou-Wu Zhang},
      year={2021},
      eprint={2105.13587},
      archivePrefix={arXiv},
      primaryClass={math.NT}
}

@book {lazarsfeld2017positivity,
    AUTHOR = {Lazarsfeld, Robert},
     TITLE = {Positivity in algebraic geometry. {I}},
    SERIES = {Ergebnisse der Mathematik und ihrer Grenzgebiete. 3. Folge. A
              Series of Modern Surveys in Mathematics [Results in
              Mathematics and Related Areas. 3rd Series. A Series of Modern
              Surveys in Mathematics]},
    VOLUME = {48},
      NOTE = {Classical setting: line bundles and linear series},
 PUBLISHER = {Springer-Verlag, Berlin},
      YEAR = {2004},
     PAGES = {xviii+387},
      ISBN = {3-540-22533-1},
   MRCLASS = {14-02 (14C20)},
  MRNUMBER = {2095471},
MRREVIEWER = {Mihnea Popa},
       DOI = {10.1007/978-3-642-18808-4},
       URL = {https://doi.org/10.1007/978-3-642-18808-4},
}

@article {Bouckdiv,
    AUTHOR = {Boucksom, S\'{e}bastien and Favre, Charles and Jonsson, Mattias},
     TITLE = {Differentiability of volumes of divisors and a problem of
              {T}eissier},
   JOURNAL = {J. Algebraic Geom.},
  FJOURNAL = {Journal of Algebraic Geometry},
    VOLUME = {18},
      YEAR = {2009},
    NUMBER = {2},
     PAGES = {279--308},
      ISSN = {1056-3911},
   MRCLASS = {14C20 (14C17)},
  MRNUMBER = {2475816},
MRREVIEWER = {James McKernan},
       DOI = {10.1090/S1056-3911-08-00490-6},
       URL = {https://doi.org/10.1090/S1056-3911-08-00490-6},
}

@article {bigyuan,
    AUTHOR = {Yuan, Xinyi},
     TITLE = {Big line bundles over arithmetic varieties},
   JOURNAL = {Invent. Math.},
  FJOURNAL = {Inventiones Mathematicae},
    VOLUME = {173},
      YEAR = {2008},
    NUMBER = {3},
     PAGES = {603--649},
      ISSN = {0020-9910},
   MRCLASS = {14G40 (11G50 14C20)},
  MRNUMBER = {2425137},
MRREVIEWER = {Antoine Chambert-Loir},
       DOI = {10.1007/s00222-008-0127-9},
       URL = {https://doi.org/10.1007/s00222-008-0127-9},
}

@article {chendiff,
    AUTHOR = {Chen, Huayi},
     TITLE = {Differentiability of the arithmetic volume function},
   JOURNAL = {J. Lond. Math. Soc. (2)},
  FJOURNAL = {Journal of the London Mathematical Society. Second Series},
    VOLUME = {84},
      YEAR = {2011},
    NUMBER = {2},
     PAGES = {365--384},
      ISSN = {0024-6107},
   MRCLASS = {14G40},
  MRNUMBER = {2835335},
MRREVIEWER = {Yuichiro Takeda},
       DOI = {10.1112/jlms/jdr011},
       URL = {https://doi.org/10.1112/jlms/jdr011},
}

@article {asympint,
    AUTHOR = {Ein, Lawrence and Lazarsfeld, Robert and Musta\c{t}\u{a}, Mircea and
              Nakamaye, Michael and Popa, Mihnea},
     TITLE = {Restricted volumes and base loci of linear series},
   JOURNAL = {Amer. J. Math.},
  FJOURNAL = {American Journal of Mathematics},
    VOLUME = {131},
      YEAR = {2009},
    NUMBER = {3},
     PAGES = {607--651},
      ISSN = {0002-9327},
   MRCLASS = {14C20},
  MRNUMBER = {2530849},
MRREVIEWER = {Tomasz Szemberg},
       DOI = {10.1353/ajm.0.0054},
       URL = {https://doi.org/10.1353/ajm.0.0054},
}

@article {lazmus,
    AUTHOR = {Lazarsfeld, Robert and Musta\c{t}\u{a}, Mircea},
     TITLE = {Convex bodies associated to linear series},
   JOURNAL = {Ann. Sci. \'{E}c. Norm. Sup\'{e}r. (4)},
  FJOURNAL = {Annales Scientifiques de l'\'{E}cole Normale Sup\'{e}rieure.
              Quatri\`eme S\'{e}rie},
    VOLUME = {42},
      YEAR = {2009},
    NUMBER = {5},
     PAGES = {783--835},
      ISSN = {0012-9593,1873-2151},
   MRCLASS = {14C20 (14E05)},
  MRNUMBER = {2571958},
MRREVIEWER = {Zach\ Teitler},
       DOI = {10.24033/asens.2109},
       URL = {https://doi.org/10.24033/asens.2109},
}

@article {positivoriginal,
    AUTHOR = {Boucksom, S\'{e}bastien and Demailly, Jean-Pierre and
              P\u{a}un, Mihai and Peternell, Thomas},
     TITLE = {The pseudo-effective cone of a compact {K}\"{a}hler manifold
              and varieties of negative {K}odaira dimension},
   JOURNAL = {J. Algebraic Geom.},
  FJOURNAL = {Journal of Algebraic Geometry},
    VOLUME = {22},
      YEAR = {2013},
    NUMBER = {2},
     PAGES = {201--248},
      ISSN = {1056-3911,1534-7486},
   MRCLASS = {14E99 (32J18 32L05 53C26)},
  MRNUMBER = {3019449},
MRREVIEWER = {Thomas\ Eckl},
       DOI = {10.1090/S1056-3911-2012-00574-8},
       URL = {https://doi.org/10.1090/S1056-3911-2012-00574-8},
}

@misc{biswas2023convex,
      title={Convex Bodies associated to Linear series of Adelic Divisors on Quasi-Projective Varieties}, 
      author={Debam Biswas},
      year={2023},
      eprint={2301.08120},
      archivePrefix={arXiv},
      primaryClass={math.AG}
}

@book {chenmoriwaki,
    AUTHOR = {Chen, Huayi and Moriwaki, Atsushi},
     TITLE = {Arakelov geometry over adelic curves},
    SERIES = {Lecture Notes in Mathematics},
    VOLUME = {2258},
 PUBLISHER = {Springer, Singapore},
      YEAR = {[2020] \copyright 2020},
     PAGES = {xviii+450},
      ISBN = {978-981-15-1727-3; 978-981-15-1728-0},
   MRCLASS = {14G40 (11G35 11G50 37P30)},
  MRNUMBER = {4292529},
MRREVIEWER = {Shu\ Kawaguchi},
       DOI = {10.1007/978-981-15-1728-0},
       URL = {https://doi.org/10.1007/978-981-15-1728-0},
}

@article {localcansaviour,
    AUTHOR = {Gubler, Walter},
     TITLE = {Local and canonical heights of subvarieties},
   JOURNAL = {Ann. Sc. Norm. Super. Pisa Cl. Sci. (5)},
  FJOURNAL = {Annali della Scuola Normale Superiore di Pisa. Classe di
              Scienze. Serie V},
    VOLUME = {2},
      YEAR = {2003},
    NUMBER = {4},
     PAGES = {711--760},
      ISSN = {0391-173X,2036-2145},
   MRCLASS = {14G40 (11G50 14G22)},
  MRNUMBER = {2040641},
MRREVIEWER = {Robin\ de Jong},
}

@article {saviourpiecewiselinear,
    AUTHOR = {Gubler, Walter},
     TITLE = {Local heights of subvarieties over non-{A}rchimedean fields},
   JOURNAL = {J. Reine Angew. Math.},
  FJOURNAL = {Journal f\"{u}r die Reine und Angewandte Mathematik. [Crelle's
              Journal]},
    VOLUME = {498},
      YEAR = {1998},
     PAGES = {61--113},
      ISSN = {0075-4102,1435-5345},
   MRCLASS = {14G20 (11G25 14D15 14G40)},
  MRNUMBER = {1629925},
MRREVIEWER = {J\"{o}rg\ Jahnel},
       DOI = {10.1515/crll.1998.054},
       URL = {https://doi.org/10.1515/crll.1998.054},
}

@article {BermanBoucksom,
    AUTHOR = {Berman, Robert and Boucksom, S\'{e}bastien},
     TITLE = {Growth of balls of holomorphic sections and energy at
              equilibrium},
   JOURNAL = {Invent. Math.},
  FJOURNAL = {Inventiones Mathematicae},
    VOLUME = {181},
      YEAR = {2010},
    NUMBER = {2},
     PAGES = {337--394},
      ISSN = {0020-9910,1432-1297},
   MRCLASS = {32L05 (32L10 32U15 32W20 58J52)},
  MRNUMBER = {2657428},
MRREVIEWER = {Norman\ Levenberg},
       DOI = {10.1007/s00222-010-0248-9},
       URL = {https://doi.org/10.1007/s00222-010-0248-9},
}

@article{GilletSoule,
     author = {Gillet, Henri and Soul\'e, Christophe},
     title = {Arithmetic intersection theory},
     journal = {Publications Math\'ematiques de l'IH\'ES},
     pages = {93--174},
     publisher = {Institut des Hautes \'Etudes Scientifiques},
     volume = {72},
     year = {1990},
     zbl = {0741.14012},
     mrnumber = {1087394},
     language = {en},
     url = {http://www.numdam.org/item/PMIHES_1990__72__93_0/}
}

@article {anabdivisor,
    AUTHOR = {Botero, Ana Mar\'{i}a and Gil, Jos\'{e} Ignacio Burgos},
     TITLE = {Toroidal b-divisors and Monge-Amp\`ere measures},
   JOURNAL = {Math. Z.},
  FJOURNAL = {Mathematische Zeitschrift},
    VOLUME = {300},
      YEAR = {2022},
    NUMBER = {1},
     PAGES = {579--637},
      ISSN = {0025-5874,1432-1823},
   MRCLASS = {14C17 (14T15 26B25 32W20)},
  MRNUMBER = {4359537},
MRREVIEWER = {Xia\ Liao},
       DOI = {10.1007/s00209-021-02789-5},
       URL = {https://doi.org/10.1007/s00209-021-02789-5},
}

@article {bdivprim,
    AUTHOR = {Shokurov, V. V.},
     TITLE = {Prelimiting flips},
   JOURNAL = {Tr. Mat. Inst. Steklova},
  FJOURNAL = {Trudy Matematicheskogo Instituta Imeni V. A. Steklova},
    VOLUME = {240},
      YEAR = {2003},
     PAGES = {82--219},
      ISSN = {0371-9685},
   MRCLASS = {14E30},
  MRNUMBER = {1993750},
MRREVIEWER = {S\'{a}ndor\ J.\ Kov\'{a}cs},
}

@book {berkovich,
    AUTHOR = {Berkovich, Vladimir G.},
     TITLE = {Spectral theory and analytic geometry over non-{A}rchimedean
              fields},
    SERIES = {Mathematical Surveys and Monographs},
    VOLUME = {33},
 PUBLISHER = {American Mathematical Society, Providence, RI},
      YEAR = {1990},
     PAGES = {x+169},
      ISBN = {0-8218-1534-2},
   MRCLASS = {32P05 (32C15 32C37 46S10 47S10)},
  MRNUMBER = {1070709},
MRREVIEWER = {W.\ Bartenwerfer},
       DOI = {10.1090/surv/033},
       URL = {https://doi.org/10.1090/surv/033},
}
